\documentclass[leqno]{article}
\pdfoutput=1
\usepackage[cp1250]{inputenc}
\usepackage{authblk}
\usepackage[english]{babel}

\usepackage{a4wide}
 
\usepackage{amsfonts}
\usepackage{graphicx}
\usepackage{url}
\usepackage{subfig}
\usepackage{amsthm} 
\usepackage{amsmath}
\usepackage{todonotes}
\newcommand{\grad}{\nabla}

\newtheorem{lemma}{Lemma}
 \newtheorem{remark}{Remark}

\title{\textbf{Hybrid coupling of finite element and boundary element methods using Nitsche's method and the Calderon projection
}} \author[1]{Timo Betcke \thanks{email: t.betcke@ucl.ac.uk}}
\author[1]{Micha\l{} Bosy \thanks{email: m.bosy@ucl.ac.uk}}
\author[1]{Erik Burman \thanks{email: e.burman@ucl.ac.uk}}
\affil[1]{\small \textit{Department of Mathematics, University College London, 25 Gordon Street, WC1H 0AY, London, United Kingdom}}

\date{\today}

\providecommand{\keywords}[1]{\textbf{Key words.} #1}

\begin{document}
\maketitle

\begin{abstract}
In this paper we discuss a hybridised method for FEM-BEM coupling. The coupling from both sides use a Nitsche type approach to couple to the trace variable. This leads to a formulation that is robust and flexible with respect to approximation spaces and can easily be combined as a building block with other hybridised methods. Energy error norm estimates and the convergence of Jacobi iterations are proved and the performance of the method is illustrated on some computational examples.

% \PACS{PACS code1 \and PACS code2 \and more}
% \subclass{MSC code1 \and MSC code2 \and more}
\end{abstract}

\section{Introduction}

The coupling of finite element (FEM) and boundary element (BEM) methods is the most widely used approach for solving multi-physical problems on an unbounded domain. It allows to take advantage of both methods. On the one hand, the BEM reduces the dimension of the problem by using the boundary integral equation, hence it is commonly used in exterior unbounded domains. On the other hand, the FEM is known from their robustness and universal applicability even for problem of inhomogeneous or non-linear nature.

The first coupled procedure was introduced by Zienkiewicz, Kelly and Bettess \cite{MR451784}. It has been analysed by Brezzi, Johnson and N\'ed\'elec~\cite{MR566158}, \cite{MR569615} and~\cite{MR583487} for problem in unbounded domains. It is often referred to as the Johnson-N\'ed\'elec coupling. Extension for higher order equations was considered by Wendland~\cite{MR902859}. The convergence analysis requires compactness of the double layer potential that can be obtain on the smooth boundary. Furthermore, even for a symmetric discretisation
scheme, the coupling method produces a system of equations with a non-symmetric coefficient matrix.

In order to avoid these disadvantages, a symmetric coupling of FEM and BEM was devised by Costabel~\cite{MR965328} and Han~\cite{MR1299224}. The independence of the compactness condition was obtain by using both equations of the Calder\'on system, contrary to the previously introduced methods that employ only one of the two equations of the Calder\'on system. Some years later, Sayas~\cite{MR2551202} showed that the weaker assumption of a Lipschitz coupling interface is sufficient for the Johnson-N\'ed\'elec coupling. His analysis has since been simplified by Steinbach~\cite{MR2831059}.

More recent developments have focused on the coupling of BEM with mixed FEM. In~\cite{MR1421985} and~\cite{MR1391614} the authors analysed symmetric coupling of BEM and mixed FEM that uses Raviart-Thomas elements. Further work on coupling BEM with mixed FEM with such elements as Brezzi-Douglas-Marini or Brezzi-Douglas-Marini-Fortin spaces was analysis by Carstensen and Funken~\cite{MR1773269}. 

Gatica, Heuer and Sayas~\cite{MR2240630} and~\cite{MR2629997} introduced the first the coupling of BEM and discontinuous Galerkin (DG) methods, in order to exploit the possibility to easily use high order approximation in the latter. Another coupling of interior penalty DG methods with BEM was presented by Of, Rodin, Steinbach and Taus~\cite{MR2970268}. 
A general approach using the unified hybridisation technique was presented by Cockburn and Sayas~\cite{MR2954729}. The class of FEM considered includes the mixed, the DG and the hybridisable discontinuous Galerkin (HDG) methods. Further collaboration of these authors with G\'uzman led to a new convergence results published in~\cite{MR3022242}.

In this paper, we present the coupling of FEM and BEM using weak imposition of coupling conditions. Nitsche's method~\cite{MR341903} is widely used in the context of FEM for imposing boundary conditions. In addition, methods based on Nitsche's method have been successfully utilized for BEM domain decomposition problems in~\cite{MR2561551} and~\cite{MR2945613}, and more recently for weakly imposing boundary conditions for BEM in~\cite{MR3945809}. 
Merging these two approaches for FEM and BEM we weakly impose both coupling conditions in a hybridised formulation on the boundary. The hybridisation is made by introducing a trace variable and imposing the coupling  in the form of a Nitsche type Dirichlet condition on the two systems. The use of Nitsche's method allows us to use the Dirichlet trace as the hybrid variable, ensuring continuity by a consistent penalty term.  The test function partner of the trace variable, then acts to ensure continuity of fluxes. The global system can be constructed using arbitrary approximation orders for the two sub systems and the trace variable and the sub problem can be solved independently. The stability of the method poses no constraint on the approximation spaces and mesh refinement does not require special treatment as in the case of Johnson-N\'ed\'elec coupling. This means that the two systems can have independent meshes, that both have to be integrated only against the trace variable. 
We here consider the standard continuous FEM, but the  formulation is by and large agnostic to the choice of FEM used in the bulk and the method can be applied with discontinuous FEM as well, such as DG, HDG~\cite{MR2485455}, or HHO~\cite{MR3259024} using a hybridised coupling on the interior domain boundary. In the case of using discontinuous FEM, our formulation can be interpreted as a hybridised interior penalty formulation of the class of methods discussed in~\cite{MR2954729}. Finally we note that, thanks to the use of Nitsche type mortaring, the method proposed herein can be used in the framework for unfitted hybridised methods introduced in \cite{MR4021278}. In that case a surface mesh is required for the definition of the BEM method, but the FEM approximation on the interior domain can be computed on an unfitted bulk discretisation.

As many existing approaches of coupling FEM and BEM, we use Finite Element Tearing And Interconnecting (FETI) and Boundary Element Tearing And Interconnecting (BETI) type of methods~\cite{MR2235732} to solve the reduced system for the hybrid variable. FETI is formulated using a Schur complement formulation, while BETI is usually formulated in terms of Steklov-Poincar\'e operators. Although Nitsche's method is an established framework for domain decomposition for finite elements methods such as FETI, it was not recognised by BETI community. In this paper, we demonstrate how the hybrid Nitsche approach can be integrated into the  BETI framework.

The rest of the paper is organized as follows. We introduce the model problem in this section. In Section~\ref{sec:var_form}, we present on continuous level the symmetric coupling of BEM and FEM formulation known from~\cite{MR965328} and~\cite{MR1299224}. The discrete formulation including weakly imposed coupling condition is introduced and analysed in Section~\ref{sec:discrete_formulation}. Although the formulation obtained is not symmetric, we comment of  how symmetry can be obtained. We discuss iterative domain decomposition in the model case of a simple relaxed Jacobi algorithm in Section~\ref{sec:iterative} and prove its convergence. In Section~\ref{sec:numerics}, we present some numerical results and we conclude with some remarks in Section~\ref{sec:conclusions}.

\subsection{Model problem}
\label{sec:notation}

% and the following problem
%\begin{equation}
%\label{eq:external_problem}
%    \left\{
%        \begin{array}{rcll}
%            - \Delta u & = & f & \mbox{in }  \Omega, \\
%             |u| & \rightarrow & 0 & \mbox{while } |x| \rightarrow %\infty.
%        \end{array}
%    \right.
%\end{equation}
Let us consider the unbounded domain $\Omega = \mathbb{R}^3$.
We divide $\Omega$ into a bounded internal part $\Omega^-$ and an unbounded external part $\Omega^+$ with common Lipschitz boundary $\Gamma$, with $n$ the outer unit normal vector of the domain $\Omega^{-}$ on $\Gamma$. %Let $\Omega^-$ be a Lipschitz domain and contain the support of $f$. 
%For $n$ the outer unit normal vector of the domain $\Omega^{-}$ on $\Gamma$, we denote $\partial_n u:= \tfrac{\partial u}{\partial n}$. 
We let $\partial_n u := \tfrac{\partial u}{\partial_n}$ denote the outward normal derivative, $f \in L^2(\Omega)$ be a function with support in $\Omega^-$ and introduce a function $\varepsilon\in L^\infty(\Omega)$, $\epsilon \ge 0$. Then we can formulate our model problem as follows
\begin{equation}
\label{eq:divided_problem}
    \left\{
        \begin{array}{rcll}
            - \Delta u^- + \varepsilon u& = & f & \mbox{in } \Omega^-, \\
            - \Delta u^+ & = & 0 & \mbox{in } \Omega^+, \\
             u^- & = & u^+  & \mbox{on } \Gamma, \\
             \partial_n u^- & = & \partial_n u^+ & \mbox{on } \Gamma, \\
             |u^+| & \rightarrow & 0 & \mbox{while } |x| \rightarrow \infty.
        \end{array}
    \right.
\end{equation}
\thanks{
where $u^i = u|_{\Omega^i}$ , $i \in \{-, +\}$. The function $\varepsilon$ is introduced to make the interior problem heterogeneous and hence unsuitable for treatment using the boundary element method alone.

\begin{remark}
It is straightforward to extend the discussion to the case with a smoothly varying diffusion coefficient in $\Omega^-$ which has a jump over $\Gamma$.
\end{remark}
%%%%%%%%%%%%%%%%%%%%%%%%%%%%%%%%%%%%%%%%%%
% Continuous problem
%%%%%%%%%%%%%%%%%%%%%%%%%%%%%%%%%%%%%%%%%%

\section{Variational formulation}
\label{sec:var_form}

Let $\left<\cdot,\cdot \right>_{\Gamma}$  denote the $L^2(\Gamma)$-inner product that can be extended to a duality pairing on $H^{-\frac{1}{2}}(\Gamma) \times H^{\frac{1}{2}}(\Gamma)$. 
We recall the following result.
\begin{lemma}[Duality pairing relation]
For any $\lambda \in H^{-\frac{1}{2}}(\Gamma)$ and $v \in H^{\frac{1}{2}}(\Gamma)$, the following relation holds
\begin{equation}
\label{eq:duality_pairing}
\left<\lambda,v\right>_{\Gamma} \leq \left\|\lambda \right\|_{H^{-\frac{1}{2}}(\Gamma)} \left\|v \right\|_{ H^{\frac{1}{2}}(\Gamma)}.
\end{equation}
\end{lemma}
\begin{proof}
It is obvious for $v = 0$ and immediately by the definition of the dual norm
\begin{equation*}
\left\|\lambda \right\|_{H^{-\frac{1}{2}}(\Gamma)} = \sup_{0 \neq v \in H^{\frac{1}{2}}(\Gamma)} \tfrac{\left<\lambda,v\right>_{\Gamma}}{\left\|v \right\|_{ H^{\frac{1}{2}}(\Gamma)}}.
\end{equation*}
\end{proof}
We start with the variational formulation of the internal problem. Applying integration by parts for first equation of~\eqref{eq:divided_problem} for every $v \in H_0^1(\Omega^-)$ we have
\begin{equation}
\label{eq:fem}
 \int_{\Omega^-} \grad u \cdot \grad v ~dx +  \int_{\Omega^-} \varepsilon u  v ~dx - \left<\partial_n u, \ v \right>_{\Gamma} = \int_{\Omega^-} f v ~dx.
\end{equation}
We define the Green's function for the Laplace operator in $\mathbb{R}^3$ as follows
\begin{equation*}
% \label{eq:green_function}
 G(x,y) := \tfrac{1}{4 \pi | x - y |}.
\end{equation*}
In this paper, we focus on the problem in $\mathbb{R}^3$. A similar analysis can be used for problems in $\mathbb{R}^2$, in which case this definition should be replaced by $G(x,y) := \frac{\log{| x - y |}}{2\pi}$.
Following the standard approach (see, e.g.~\cite[Chapter~6]{MR2361676}), we introduce single layer and double layer operators $\mathcal{V}:H^{-\frac{1}{2}}(\Gamma) \rightarrow H^1(\Omega^+)$ and $\mathcal{K}:H^{\frac{1}{2}}(\Gamma) \rightarrow H^1(\Omega^+)$ respectively as
\begin{align*}
% \label{eq:single_layer_op}
 (\mathcal{V} \varphi) (x) &:= \int_{\Gamma} G(x,y) \varphi(y)~ dy && \mbox{for } \varphi \in H^{-\frac{1}{2}}(\Gamma), \\
% \label{eq:double_layer_op}
 (\mathcal{K} v) (x) &:= \int_{\Gamma} \tfrac{\partial G(x,y)}{\partial n_y} v(y) ~dy && \mbox{for } v \in H^{\frac{1}{2}}(\Gamma),
\end{align*}
where $x \in \Omega^+ \setminus \Gamma$ and $n_y$ is an outer unit normal vector (for $\Omega^i$, $i \in \{-, +\}$) in the point $y$.

Following~\cite[Chapter~1]{MR2361676}, we define the Dirichlet and Neumann traces
\begin{align*}
% \label{eq:trace_map}
\gamma_D^i: H^1(\Omega^i) &\rightarrow H^{\frac{1}{2}}(\Gamma) & \gamma_D^i f(x) & := \lim_{\Omega^i \ni y \rightarrow x \in \Gamma}  f(y), \\
\gamma_N^i: H^1(\Delta,\Omega^i) &\rightarrow H^{-\frac{1}{2}}(\Gamma) & \gamma_N^i f(x) & := \lim_{\Omega^i \ni y \rightarrow x \in \Gamma}  n_x \cdot \grad f(y),
\end{align*}
where $H^1(\Delta,\Omega^i) := \left\{v \in H^1(\Omega^i): \Delta v \in L^2(\Omega^i)\right\}$, for $i \in \{-, +\}$, and $n_x$ is an outer (for $\Omega^-$) normal vector to $\Gamma$ in the point $x$. 
 The following results will be  useful in what follows.
\begin{lemma}[Trace theorem]
Let $i \in \{-, +\}$, then for $\Omega^i \subset \mathbb{R}^3$ the trace operator $\gamma_D^i: H^1(\Omega^i) \rightarrow H^{\frac{1}{2}}(\Gamma)$ is bounded for all $v \in H^1(\Omega^i)$
\begin{equation}
\label{eq:trace_theorem}
 \|\gamma_D^i v \|_{H^{\frac{1}{2}}(\Gamma)} \leq C_T \|v \|_{H^1(\Omega^i)}.
\end{equation}
%Furthermore, there exists an extension operator $\mathcal{E}: H^{\frac{1}{2}}(\Gamma) \rightarrow H^1(\Omega^i)$ satisfying $\gamma_D^i \mathcal{E} w = w$ for %all $w \in H^{\frac{1}{2}}(\Gamma)$. The following bound bounds.
%\begin{equation}
%\label{eq:inverse_trace_theorem}
% \|\mathcal{E} w \|_{H^1(\Omega^i)} \leq C_{IT} \|w \|_{H^{\frac{1}{2}}(\Gamma)}.
%\end{equation}
\end{lemma}
\begin{proof}
See~\cite[Theorem 2.21]{MR2361676}.
\end{proof}
We use $\{ \cdot \}_{\Gamma}$ to denote an average of the interior and exterior traces of a function. Then, applying the trace mappings yields to the single layer, double layer, adjoint double layer potentials and hypersingular boundary integral operator
\begin{align*}
 V: H^{-\frac{1}{2}}(\Gamma) & \rightarrow H^{\frac{1}{2}}(\Gamma) & V &:= \{\gamma_D \mathcal{V}\}_{\Gamma}, \\
 K: H^{\frac{1}{2}}(\Gamma) & \rightarrow H^{\frac{1}{2}}(\Gamma) & K &:= \{\gamma_D \mathcal{K}\}_{\Gamma}, \\
 K': H^{-\frac{1}{2}}(\Gamma) & \rightarrow H^{-\frac{1}{2}}(\Gamma) & K' &:= \{\gamma_N \mathcal{V}\}_{\Gamma}, \\
 W: H^{\frac{1}{2}}(\Gamma) & \rightarrow H^{-\frac{1}{2}}(\Gamma) & W &:= \{\gamma_N \mathcal{K}\}_{\Gamma}.
\end{align*}
For the solution $u$ of the problem~\eqref{eq:divided_problem}, we have the following boundary integral equations on $\Gamma$
\begin{align}
\label{eq:bound_integral_eq}
\begin{pmatrix}
\gamma_D^-u \\
\gamma_N^-u 
\end{pmatrix}
= C^-
\begin{pmatrix}
\gamma_D^-u \\
\gamma_N^-u 
\end{pmatrix}
&, & 
\begin{pmatrix}
\gamma_D^+u \\
\gamma_N^+u 
\end{pmatrix}
= C^+
\begin{pmatrix}
\gamma_D^+u \\
\gamma_N^+u 
\end{pmatrix} ,
\end{align}
where $C^{\pm}: H^{\frac{1}{2}}(\Gamma) \times  H^{-\frac{1}{2}}(\Gamma) \rightarrow H^{\frac{1}{2}}(\Gamma) \times  H^{-\frac{1}{2}}(\Gamma)$ denotes two Calder\'{o}n projectors defined as follows
\begin{equation*}
% \label{eq:calderon_proj}
 C^{\pm} := 
\begin{pmatrix}
\tfrac{1}{2} Id \pm K & \mp V \\
\mp W & \tfrac{1}{2} Id \mp K'
\end{pmatrix}.
\end{equation*}
From the relation~\eqref{eq:bound_integral_eq} for external traces we can construct the following exterior Dirichlet-to-Neumann operator
\begin{equation}
 \label{eq:dtn_op}
  DtN^+ := -W+(\tfrac{1}{2} Id-K') \circ V^{-1} \circ (K-\tfrac{1}{2} Id).
\end{equation}
Obviously, it makes sense only if the inverse of the operator $V$ exists.

Using  Dirichlet-to-Neumann operator~\eqref{eq:dtn_op} we introduce a new variable $\lambda = \gamma_N^+u = \partial_n u^+$ as 
\begin{equation}
 \label{eq:lambda}
 \lambda := \left(V^{-1} \circ (K-\tfrac{1}{2} Id) \right) \gamma_D^- u.
\end{equation}

The classical symmetric coupling that satisfies the transmission conditions of~\eqref{eq:divided_problem} is as follow
\begin{center}
  \textit{Find $u \in H^1(\Omega^-)$ and $\lambda \in H^{-\frac{1}{2}}(\Gamma)$ such that for all $v \in H^1(\Omega^-)$ and $\zeta \in H^{-\frac{1}{2}}(\Gamma)$}
\begin{equation} 
 \label{eq:fem_bem_cont}
 \left\{
 \begin{array}{rcl}
  \int_{\Omega^-} \grad u \cdot \grad v ~dx +  \int_{\Omega^-} \varepsilon u  v ~ dx + \left<W u, \ v \right>_{\Gamma} - \left< \left(\tfrac{1}{2} Id - K'\right) \lambda, \ v \right>_{\Gamma} &=& \int_{\Omega^-} f v ~dx, \\[3mm] 
  \left<\left(\tfrac{1}{2} Id - K\right) u, \ \zeta \right>_{\Gamma} + \left<  V \lambda, \ \zeta \right>_{\Gamma} &=&0.
  \end{array}
  \right.
\end{equation}
\end{center}

\subsection{Well-posedness of the continuous problem}
\label{sec:cont_exis_uniq}
The following results are well known (see~\cite{MR965328} or~\cite{MR1299224}), but for reader's convenience we present them in the case of problem~\eqref{eq:divided_problem}.
Let us propose a more compact formulation of~\eqref{eq:fem_bem_cont}.
\begin{center}
  \textit{Find $u \in H^1(\Omega^-)$ and $\lambda \in H^{-\frac{1}{2}}(\Gamma)$ such that for all $v \in H^1(\Omega^-)$ and $\zeta \in H^{-\frac{1}{2}}(\Gamma)$}
\begin{equation} 
 \label{eq:fem_bem_cont_compact}
 A\left(\left(u, \lambda \right), \left(v, \zeta\right)\right) = \int_{\Omega^-} f v ~dx,
\end{equation}
\end{center}
where
\begin{align*}
 A\left(\left(w, \lambda \right), \left(v, \zeta\right)\right) :=& \int_{\Omega^-} \grad w \cdot \grad v~ dx +  \int_{\Omega^-} \varepsilon u  v ~dx - \tfrac{1}{2} \left< \lambda, \ v \right>_{\Gamma} + \tfrac{1}{2}\left< w, \ \zeta \right>_{\Gamma} \\
 &+ \left<W w, \ v \right>_{\Gamma} + \left< K' \lambda, \ v \right>_{\Gamma} - \left< K w, \ \zeta \right>_{\Gamma} + \left<  V \lambda, \ \zeta \right>_{\Gamma}.
\end{align*}
For simplicity we introduce the space $\mathbb{V} := H^1(\Omega^-) \times H^{-\frac12}(\Gamma)$ and the associated norm
 \begin{equation}
  \label{eq:norm_continuous}
  \left\|(v,\varphi) \right\|_{\mathbb{V}}^2 := \left\|v \right\|_{ H^1(\Omega^-)}^2 + \left\|\varphi \right\|_{ H^{-\frac{1}{2}}(\Gamma)}^2.
 \end{equation}

\begin{lemma}[Continuity]
  \label{l:continuity_continuous}
 There exists constant $\beta > 0$ such that for all $w,v \in H^1(\Omega^-)$ and $\lambda, \varphi \in H^{-\frac12}(\Gamma)$
 \begin{equation}
  \label{eq:continuity_continuous}
  \left|A\left(\left(w, \lambda \right), \left(v, \varphi\right)\right)\right| \leq \beta \left\|(w,\lambda) \right\|_{\mathbb{V}} \left\|(v,\varphi) \right\|_{\mathbb{V}}.
 \end{equation}
\end{lemma}
\begin{proof}
 We use the Cauchy-Schwarz inequality, the relation~\eqref{eq:duality_pairing} and continuity of boundary operators (see~\cite[Section~6.2-6.5]{MR2361676}) to obtain
  \begin{align*}
  \left|A\left(\left(w, \lambda \right), \left(v, \varphi\right)\right)\right| \leq& \max\{1,\|\varepsilon\|_{L^\infty(\Omega)}\} \left\|w \right\|_{ H^1(\Omega^-)} \left\|v \right\|_{ H^1(\Omega^-)}\\
   & + \tfrac{1}{2} \left\|w \right\|_{ H^{\frac{1}{2}}(\Gamma)} \left\|\varphi \right\|_{ H^{-\frac{1}{2}}(\Gamma)} + \tfrac{1}{2} \left\|\lambda \right\|_{ H^{-\frac{1}{2}}(\Gamma)} \left\|v \right\|_{ H^{\frac{1}{2}}(\Gamma)} \\
   & + C_K \left\|w \right\|_{ H^{\frac{1}{2}}(\Gamma)} \left\|\varphi \right\|_{ H^{-\frac{1}{2}}(\Gamma)} + C_V \left\|\lambda \right\|_{ H^{-\frac{1}{2}}(\Gamma)} \left\|\varphi \right\|_{ H^{-\frac{1}{2}}(\Gamma)} \\
   & + C_{K'} \left\|\lambda \right\|_{ H^{-\frac{1}{2}}(\Gamma)} \left\|v \right\|_{ H^{\frac{1}{2}}(\Gamma)} + C_{W} \left\|w \right\|_{ H^{\frac{1}{2}}(\Gamma)} \left\|v \right\|_{ H^{\frac{1}{2}}(\Gamma)}.
 \end{align*}
 We finished by applying the trace inequality~\eqref{eq:trace_theorem} to terms including $\left\|\cdot \right\|_{H^{\frac{1}{2}}(\Gamma)}$ norm, 
$$
\left\|v \right\|_{ H^{\frac{1}{2}}(\Gamma)} +\left\|w \right\|_{ H^{\frac{1}{2}}(\Gamma)}  \leq C_T \left(\left\|v \right\|_{ H^1(\Omega^-) }+ \left\|w \right\|_{ H^1(\Omega^-)}\right).
$$
\end{proof}

\begin{lemma}[Coercivity]
  \label{l:coercivity_continuous}
 There exists constant $\alpha > 0$ such that for all $v \in H^1(\Omega^-)$ and $\varphi \in H^{-\frac12}(\Gamma)$
 \begin{equation}
  \label{eq:coercivity_continuous}
  A\left(\left(v, \varphi \right), \left(v, \varphi \right)\right) \geq \alpha  \left\|(v,\varphi) \right\|_{\mathbb{V}}^2.
 \end{equation}
\end{lemma}
\begin{proof}
%%% We do not have the Poincaré inequality!
 Using coercivity of $V$ (see~\cite[Theorem~6.22]{MR2361676}) and coercivity of $W$ (see~\cite[Theorem~6.24]{MR2361676}) we obtain
  \begin{align*}
  A\left(\left(v, \varphi \right), \left(v, \varphi \right)\right) \geq & c_\epsilon \left\|v \right\|_{ H^1(\Omega^-)}^2 + \alpha_V \left\|\varphi \right\|_{ H^{-\frac{1}{2}}(\Gamma)}^2 + \alpha_W \left\|v - \bar v\right\|_{ H^{\frac{1}{2}}(\Gamma)}^2
 \end{align*}
where $\bar v$ denotes the average over $\Gamma$ of $v$ and $c_\varepsilon = \min(1,\varepsilon)$. 
This shows ~\eqref{eq:coercivity_continuous} when $\varepsilon>0$. For the case $\epsilon = 0$ we need a Poincar\'e inequality of the form 
\begin{equation}\label{eq:Pcare}
c_P \|v\|_{H^1(\Omega)} \leq \|\nabla v\|_{L^2(\Omega)} + |g(v)|
\end{equation}
where $g(v)$ is some functional that is non-zero for constant non-zero $v$ \cite[Lemma B.63]{MR2050138}. We claim that this holds with $g(v) = \left<  V \varphi(v), \varphi(v) \right>_{\Gamma}$ where $\varphi(v)$ is defined by the second equation of \eqref{eq:fem_bem_cont}. We immediately see that if this is true then there exists $\alpha>0$ such that
\begin{align*}
  A\left(\left(v, \varphi \right), \left(v, \varphi \right)\right) \geq & \alpha ( \left\|v \right\|_{ H^1(\Omega^-)}^2 %+  \left\|\varphi \right\|_{ H^{-\frac{1}{2}}(\Gamma)}^2 +  \left\|v - \bar v}\right\|_{H^{\frac{1}{2}}(\Gamma)^2).
 \end{align*}
We need to show that for constant $v \ne 0$ there holds $\left<  V \varphi(v), \varphi(v) \right>_{\Gamma} \ne 0$.
Let $v\vert_\Gamma = u_0$ be a non-zero constant. Then we need to study
$$
\left<\left(\tfrac{1}{2} Id - K\right) u_0, \ \zeta \right>_{\Gamma} + \left<  V \varphi(v), \ \zeta \right>_{\Gamma} =0, \quad \forall \zeta \in H^{-\frac12}(\Gamma).
$$
%%% I argue by contradiction here, but one could use a direct proof with applying property $\left<\left(\tfrac{1}{2} Id + K\right) u_0, \ \zeta \right>_{\Gamma} = 0$ directly.
% (Steinbach Proof Corollary 6.27) % It really does not matter.
We argue by contradiction. Assume that $\varphi(v) = 0$. Then $u_0$ is the trace of a solution to the homogeneous Neumann problem in $\Omega^-$.
However by the first line of the left relation of equation \eqref{eq:bound_integral_eq} there holds for all such traces % (see also \cite[Exercise 3.8.6]{MR2743235}),
%%% Maybe this citation is a bit overkill, it follows from the first line of equation (6)
%\begin{equation}\label{eq:sing_pot}
%\left<\left(\tfrac{1}{2} Id + K\right) u_0, \ \zeta \right>_{\Gamma} = 0, \quad %\forall \zeta \in H^{-\frac12}(\Gamma)
%\end{equation}
%and hence
\begin{equation}\label{eq:sing_pot}
\left<\left(\tfrac{1}{2} Id - K\right) u_0, \ \zeta \right>_{\Gamma} = \left< Id\, u_0, \ \zeta \right>_{\Gamma}, \quad \forall \zeta \in H^{-\frac12}(\Gamma).
\end{equation}
Hence the functional is defined by
\[
 \left<  V \varphi(v), \ \zeta \right>_{\Gamma}= - \left< Id\, u_0, \ \zeta \right>_{\Gamma}, \quad \forall \zeta \in H^{-\frac12}(\Gamma).
\]
However since the operator on the left hand side, defined by, $V$ is injective it follows that $ \varphi(v) \ne 0$, which leads to a contradiction. Hence $\left<  V \varphi(v), \varphi(v) \right>_{\Gamma} \ne 0$ for $v$ constant.
This concludes the proof.
\end{proof}

The existence and uniqueness of the solution of problem~\eqref{eq:fem_bem_cont_compact} is achieved by using the Lax-Milgram theorem.

%%%%%%%%%%%%%%%%%%%%%%%%%%%%%%%%%%%%%%%%%%
% Discrete problem
%%%%%%%%%%%%%%%%%%%%%%%%%%%%%%%%%%%%%%%%%%

\section{The discrete problem}
\label{sec:discrete_formulation}

%%%%%%%%%%%%%%%%%%%%%%%%%%%%%%%%%%%%%%%%%%%%%%%%%%
%% Model problem

%%%%%%%%%%%%%%%%%%%%%%%%%%%%%%%%%%%%%%%%%%%%%%%%%%
%% Triangulation
We assume that $\Omega^-$ is a polyhedral domain. The boundary $\Gamma$ may be decomposed in a set of $n$ planar surfaces $\{\Gamma_j\}_{j=1}^n$. % The outer boundary $\Gamma$ of $\Omega_h^-$ is supposed to be a polyhedral curve with nodes on $\Gamma$ which approximates $\Gamma$. 
It will be convenient to use following broken Sobolev spaces over the polyhedral boundary~$\Gamma$. For $s>1$ define %on %the triangulation $\mathcal{T}_h$ and 
% the set of all edges in $\mathcal{E}_h \cap \Gamma$
\begin{equation*}
\widetilde H^s(\Gamma)  :=  \left\{v \in H^1(\Gamma) : \ v|_{\Gamma_j} \in H^s(\Gamma_j), \, 1\leq j \leq n  \right\}.
%\\
% H^m(\mathcal{T}_h) & :=  \left\{v \in L^2(\Omega): \ {v}|_K \in H^m(K) \ \forall \ K \in \mathcal{T}_h \right\} \mbox{ for } m \in \mathbb{N},
\end{equation*}
% with the corresponding broken norms  for $s>0$
% \begin{align*}
% \|w_h\|_{H^s(\mathcal{E}_h \cap \Gamma)} := \sum_{E \in \mathcal{E}_h} \|w_h\|_{H^s(\Gamma \cap E)} %&,& \left<\widetilde{v}_h, \widetilde{v}_h\right>_{\Gamma} := \sum_{E \in \mathcal{E}_h} \left<\widetilde{v}_h, \widetilde{v}_h\right>_{\Gamma \cap E}
% \end{align*}
When $s>1$ the
norm on  $\widetilde H^s(\Gamma)$,  is defined as the broken
norm over the faces of the polyhedral boundary $\Gamma$  $$\|v\|_{\widetilde H^{s}(\Gamma)} := \left(\sum_{j=1}^n \|v\|^2_{ H^{s}(\Gamma_j)}\right)^{\frac12}.$$

When $0 \leq s \leq 1$ the space $\widetilde H^s(\Gamma)$ coincides with the usual space $H^s(\Gamma)$ and their norms are the same (for more details see~\cite[Definition 4.1.48]{MR2743235}).

Let $\mathcal{T}_h$ be a triangulation of $\overline{\Omega}^-$ made of tetrahedrons. 
For each triangulation $\mathcal{T}_h$, $\mathcal{E}_h$ denotes the set of its facets. In addition, for each 
of element $K \in \mathcal{T}_h$, $h_K := \mbox{diam}(K)$, and %for simplicity we assume that mesh is quasi-uniform, so 
$h := \max_{K \in \mathcal{T}_h} h_K$. Let $\mathcal{G}_i$, $i=1,2$ denote two different surface triangulation of the boundary $\Gamma$. For notation convenience we here assume that the trace mesh of $\mathcal{T}_h$ and the $\mathcal{G}_i$ all have the similar local mesh size.

 The following result will be useful in what follows.

%%%%%%%%%%%%%%%%%%%%%%%%%%%%%%%%%%%%%%%%%%%%%%%%%%
%% Inverse, trace inequalities
\begin{lemma}[Trace inequality]
There exists $C_{max} > 0$, independent of $h_K$, such that for all $K \in \mathcal{T}_h$ and polynomial function $v$ in $K$ the following discrete trace inequality holds
\begin{align}
% \label{eq:local_inverse}
% |v|_{H^s(K)} &\leq C_1 h_K^{m-s} |v|_{H^m(K)}, \ 0 \leq m \leq s, \\
% \label{eq:inverse_trace}
% \|v\|_{H^s(\partial K \cap \Gamma)} &\leq C_2 h_K^{m-s} \|v\|_{H^m(\partial K \cap \Gamma)}, \ m \leq s, \mbox{ and } |m|,|s| \leq 1, \\
\label{eq:discrete_trace_inequality}
h_K^{\frac{1}{2}} \|v\|_{L^2(\partial K)} &\leq C_{max} \|{v}\|_{L^2(K)}.
% \end{align} 
% Moreover, there exists $C > 0$, independent of $h$, such that for all polynomial function $v$ in $\Gamma$ the following inequality holds
% % \begin{equation}
% \label{eq:inverse_trace}
% |v|_{H^s(\Gamma)} &\leq C_2 h_K^{m-s} |v|_{H^m(\Gamma)}, \ m \leq s, \mbox{ and } |m|,|s| \leq 1.
% \end{equation} 
\end{align}
% In addition, there exists $C > 0$, independent of $h_K$, such that for any $v \in H^1(K)$ the following local trace inequality hold
% \begin{equation}
% \label{eq:local_trace}
% \|v\|_{L^2(\partial K)} \leq C \left(h_K^{-\frac{1}{2}} \|v\|_{L^2( K)} + h_K^{\frac{1}{2}} |v|_{H^1( K)}\right).
% \end{equation} 
\end{lemma}
\begin{proof}
% For the discrete trace inequality~\eqref{eq:discrete_trace_inequality}
See~\cite[Lemma 1.46]{MR2882148}
% and the local trace inequality~\eqref{eq:local_trace} follows by~\cite[Theorem 1.6.6]{MR2373954} and standard scaling arguments.
\end{proof}
%  Now we will introduce the finite element spaces that discretise the above spaces. 
 
 To discretise the problem~\eqref{eq:fem_bem_cont} over the triangulation one can choose either continuous or discontinuous finite elements. For simplicity of the analysis we choose the following spaces
 \begin{align*}
 V_{h}^{j} &:= \left\{{v}_h \in C^0\left(\Omega^-\right): \ {v}_h|_K \in \mathbb{P}_{j}\left(K \right) \ \forall \ {K \in \mathcal{T}_h}\right\}, \\
% W_{h}^{k} &:= \left\{{w}_h \in C^0\left(\Gamma\right): \ {w}_h|_E \in \mathbb{P}_{k}\left(E \right) \ \forall \ {E \in \left(\mathcal{E}_h \cap \Gamma\right)} \right\}, \\
% \Lambda_{h}^{l} &:= \left\{\lambda_h \in L^{2}\left(\Gamma\right): \ \lambda_h|_E \in \mathbb{P}_{l}\left(E \right) \ \forall \ {E \in \left(\mathcal{E}_h \cap \Gamma\right)} \right\}, \\
% M_{h}^{m} &:= \left\{\widetilde{v}_h \in H^{\frac{1}{2}}\left(\Gamma\right): \ \widetilde{v}_h|_E \in \mathbb{P}_{m}\left(E \right) \ \forall \ {E \in \left(\mathcal{E}_h \cap \Gamma\right)} \right\}. 
W_{h}^{k} &:= \left\{{w}_h \in C^0\left(\Gamma\right): \ {w}_h|_E \in \mathbb{P}_{k}\left(E \right) \ \forall \ E \in \mathcal{G}_1   \right\}, \\
 \Lambda_{h}^{l} &:= \left\{\lambda_h \in L^{2}\left(\Gamma\right): \ \lambda_h|_E \in \mathbb{P}_{l}\left(E \right) \ \forall \ E \in \mathcal{G}_1  \right\}, \\
M_{h}^{m} &:= \left\{ \widetilde{v}_h \in L^2(\Gamma):\ %\widetilde{v}_h\vert_{\Gamma_j} \in  C^0(\Gamma_j)} \mbox{ and } 
\ \widetilde{v}_h|_E \in \mathbb{P}_{m}\left(E \right) \ \forall \ E \in \mathcal{G}_2  \right\}. 
\end{align*} 

Let us denote $\mathcal{V}_h := V_{h}^{j} \times W_{h}^{k}$ and $v_h = \left(v_h^-, v_h^+\right) \in \mathcal{V}_h$.
Using the above spaces we propose the hybrid discrete formulation of the problem~\eqref{eq:fem_bem_cont}
\begin{center}
\textit{Find $\left(u_h, \lambda_h, \widetilde{u}_h\right) \in \mathcal{V}_h\times \Lambda_{h}^{l} \times M_{h}^{m}$ such that for all $\left(v_h, \varphi_h, \widetilde{v}_h\right) \in \mathcal{V}_h \times \Lambda_{h}^{l} \times M_{h}^{m}$}
\end{center}
\begin{equation} 
 \label{eq:fem_bem_disctrete}
 a_h\left(\left(u_h^-, \widetilde{u}_h \right), \left(v_h^-, \widetilde{v}_h\right)\right) + b_h\left( \left(u_h^+, \lambda_h, \widetilde{u}_h\right), \left(v_h^+, \varphi_h, \widetilde{v}_h\right)\right)  = \int_{\Omega^-} f v_h^- dx,
\end{equation}
where
\begin{align*}
a_h\left(\left(w_h, \widetilde{w}_h \right), \left(v_h, \widetilde{v}_h\right)\right) := & \int_{\Omega^-} \grad w_h : \grad v_h dx +  \int_{\Omega^-} \varepsilon  w_h v_h dx  \\
& - \left< \partial_n w_h, \ v_h - \widetilde{v}_h \right>_{\Gamma}   - \left< w_h - \widetilde{w}_h, \ \partial_n v_h \right>_{\Gamma} \\
&  + \tfrac{\tau}{h} \left< w_h - \widetilde{w}_h,\ v_h - \widetilde{v}_h \right>_{\Gamma}, \\
b_h\left( \left(w_h, \lambda_h, \widetilde{w}_h\right), \left(v_h, \varphi_h, \widetilde{v}_h\right)\right) := & \left<\left( \tfrac{1}{2}Id - K \right) w_h, \ \varphi_h \right>_{\Gamma} + \left<  V \lambda_h, \ \varphi_h \right>_{\Gamma} \\
&  + \left<W w_h, \ v_h \right>_{\Gamma}  - \left<\left( \tfrac{1}{2} Id -  K'\right)  \lambda_h, \ v_h \right>_{\Gamma}\\
&  +  \left< \lambda_h, \ v_h - \widetilde{v}_h \right>_{\Gamma} -  \left< w_h - \widetilde{w}_h, \ \varphi_h \right>_{\Gamma}\\
&  +  \tfrac{\tau}{h}  \left< w_h - \widetilde{w}_h,\ v_h - \widetilde{v}_h \right>_{\Gamma}.
\end{align*}
The stabilisation parameter $\tau >0$  has to be chosen appropriately. The formulation of bilinear form $a_h$ is well known for example from~\cite{Egger09aclass}. As we said before it is possible to use the discontinuous finite element method for example symmetric interior penalty hybrid discontinuous Galerkin method presented in~\cite{MR2727822}.  

%\begin{remark}(Non conforming domain decomposition)
%It is only to reduce the technical detail that we define all our quantities on %the same mesh $\mathcal{T}_h$. Indeed the trace variable $\widetilde u_h$ can be %defined on a different (coarser)   mesh and use higher polynomial approximation %to reduce the coupling degrees of freedom. The BEM variables $u_h^+$ and %$'lambda_h$ could also be defined on a separate mesh.
%\end{remark}

\begin{remark}[Impedance boundary condition]
\label{rmk:robin}
The hybrid weakly imposed Dirichlet and Neumann boundary conditions is related to an  impedance boundary condition of the type
\begin{equation*}
 \widetilde{u} = -\gamma \frac{\partial u}{\partial n} + u,
\end{equation*}
with $\gamma = \frac{h}{\tau}$.
This can be seen considering terms associated with $\widetilde{v}_h$ in above definition of the bilinear forms.
\end{remark}

\begin{remark}[Relation to a standard Nitsche type method without hybridisation]
The trace variable $\widetilde u_h$ can be eliminated by replacing it by a linear combination of $u_h^+$ and $u_h^-$ with similarly the test function $\widetilde v_h$ replaced by the same linear combination of the test functions $v_h^+$ and $v_h^-$  (see \cite[Section 4.2]{BHL21}). The below analysis carries over verbatim to this case. If in addition $v_h^+$ and $v_h^-$ are chosen in the same spaces (with the same trace meshes) $u_h^-$ and $v_h^-$ can be substituted for $u_h^+$ and $v_h^+$ respectively, resulting in a (fully coupled) method with the unknowns $u_h^-$ and $\lambda_h$ only.
\end{remark}
%%%%%%%%%%%%%%%%%%%%%%%%%%%%%%%%%%%%%%%%%
% Algebraic formulation
%%%%%%%%%%%%%%%%%%%%%%%%%%%%%%%%%%%%%%%%%

\subsection{Symmetric formulation}
\label{sec:symmetric}

Despite using the symmetric Nitsche method, our whole system is not symmetric. This is a consequence of the lack of symmetry of the boundary element method with weak imposition. We can use the Steklov-Poincar\'{e} operator to eliminate the flux variable, so that the non-symmetric method above is transformed into a symmetric reduced system as we show below.

The following equations are associated with bilinear form $b_h$ from~\eqref{eq:fem_bem_disctrete} reads
\begin{align*}
 -\left<\left( \tfrac{1}{2} Id - K\right) w_h, \varphi_h\right>_{\Gamma} - \left<V \lambda_h, \varphi_h\right>_{\Gamma} &= -\left<w_h - \widetilde{w}_h, \varphi_h\right>_{\Gamma}, \\
 -\left<W w_h, v_h\right>_{\Gamma} + \left<\left(\tfrac{1}{2} Id - K'\right) \lambda_h, v_h\right>_{\Gamma}
  - \left<\lambda_h, v_h - \widetilde{v}_h\right>_{\Gamma} &= \tfrac{\tau}{h} \left<w_h - \widetilde{w}_h, v_h - \widetilde{v}_h\right>_{\Gamma}.
\end{align*}
Similar to the continuous formulation we use the Dirichlet-to-Neumann operator~\eqref{eq:dtn_op} to obtain
\begin{equation}
 \label{eq:lambda_h}
 \lambda_h := \left(V^{-1} \circ (K-\tfrac{1}{2} Id) \right) w_h + V^{-1}  \left( w_h - \widetilde{w}_h \right).
\end{equation}
Injecting this relation into the second equation leads to the formulation of the new symmetric bilinear form
\begin{align*}
\widehat{b}_h\left( \left(w_h, \widetilde{w}_h\right), \left(v_h, \widetilde{v}_h\right)\right) := & \left<W w_h, \ v_h \right>_{\Gamma} - \left<\left( \tfrac{1}{2} Id -  K'\right)  V^{-1}  (K-\tfrac{1}{2} Id)w_h, \ v_h \right>_{\Gamma} \\
& - \left<\left( \tfrac{1}{2} Id -  K'\right) V^{-1}  \left( w_h - \widetilde{w}_h \right), \ v_h \right>_{\Gamma} \\
& + \left< V^{-1}  (K-\tfrac{1}{2} Id)w_h, \ v_h - \widetilde{v}_h \right>_{\Gamma}\\
&  + \left< V^{-1}  \left( w_h - \widetilde{w}_h \right), \ v_h - \widetilde{v}_h \right>_{\Gamma} + \tfrac{\tau}{h}  \left< w_h - \widetilde{w}_h,\ v_h - \widetilde{v}_h \right>_{\Gamma}.
\end{align*}

%%%%%%%%%%%%%%%%%%%%%%%%%%%%%%%%%%%%%%%%%%%%%%%%%%%%%%%%%%%%%%%%%%%%%%%%%%%%%%%%%%%%%%%%%%%%%%%%%%%%%%%%%%%%
% Well-posedness of the discrete problem
%%%%%%%%%%%%%%%%%%%%%%%%%%%%%%%%%%%%%%%%%%%%%%%%%%%%%%%%%%%%%%%%%%%%%%%%%%%%%%%%%%%%%%%%%%%%%%%%%%%%%%%%%%%%

\subsection{Well-posedness of the discrete problem}
\label{sec:exis_uniq}
Let us consider the following norms
\begin{align}
\nonumber
&&\left\|\left({w_h}, {\widetilde{w}_h}\right)\right\|_{\mathcal{F}_*}^2 := & \|{w_h}\|_{H^1(\Omega^-)}^2 + \tfrac{\tau}{h} \left\|{w_h} - {\widetilde{w}_h}\right\|_{L^2(\Gamma)}^2, \\
\label{eq:FEM_norm}
&&\left\|\left({w_h}, {\widetilde{w}_h}\right)\right\|_{\mathcal{F}}^2 := & \left\|\left({w_h}, {\widetilde{w}_h}\right)\right\|_{\mathcal{F}_*}^2 + h \left\| {\partial_n w_h}\right\|_{L^2(\Gamma)}^2, \\
\nonumber
&&\left\|\left({w_h}, \lambda_h, {\widetilde{w}_h}\right)\right\|_{\mathcal{B}_*}^2 := & \|w_h\|_{H^{\frac{1}{2}}(\Gamma)}^2 + \|\lambda_h\|_{H^{-\frac{1}{2}}(\Gamma)}^2 + \tfrac{\tau}{h} \left\|{w_h} - {\widetilde{w}_h}\right\|_{L^2(\Gamma)}^2, \\
\label{eq:BEM_norm}
&&\left\|\left({w_h}, \lambda_h, {\widetilde{w}_h} \right)\right\|_{\mathcal{B}}^2 := & \left\|\left({w_h}, \lambda_h, {\widetilde{w}_h}\right)\right\|_{\mathcal{B}_*}^2 + h \left\| \lambda_h\right\|_{L^2(\Gamma)}^2. 
\end{align}
\begin{lemma}[Equivalence of the norms]
 For all $\left({w_h}, \lambda_h, {\widetilde{w}_h}\right) \in  \mathcal{V}_h \times \Lambda_{h}^{l} \times M_{h}^{m}$ there exist positive constants $C_{\mathcal{F}}, C_{\mathcal{B}}$ such that 
 \begin{align}
 \label{eq:fem_norm_equivalence}
  \left\|\left({w_h}^-, {\widetilde{w}_h}\right)\right\|_{\mathcal{F}_*} \leq \left\|\left({w_h}^-, {\widetilde{w}_h}\right)\right\|_{\mathcal{F}} &\leq C_{\mathcal{F}} \left\|\left({w_h}^-, {\widetilde{w}_h}\right)\right\|_{\mathcal{F}_*},
  \\
   \label{eq:bem_norm_equivalence}
    \left\|\left({w_h}^+, \lambda_h, {\widetilde{w}_h}\right)\right\|_{\mathcal{B}_*} \leq \left\|\left({w_h}^+, \lambda_h, {\widetilde{w}_h}\right)\right\|_{\mathcal{B}} &\leq C_{\mathcal{B}} \left\|\left({w_h}^+, \lambda_h, {\widetilde{w}_h}\right)\right\|_{\mathcal{B}_*}.
 \end{align}
\end{lemma}
\begin{proof}
 For~\eqref{eq:fem_norm_equivalence} we use the trace inequality~\eqref{eq:discrete_trace_inequality} and for~\eqref{eq:bem_norm_equivalence} we use the inverse inequality $ h^{\frac12} \left\| \lambda_h\right\|_{L^2(\Gamma)} \leq C \left\|\lambda_h\right\|_{H^{-\frac12}(\Gamma)}$ (see~\cite[Remark 4.4.4]{MR2743235}).
\end{proof}

\begin{lemma}[Continuity]
  \label{l:continuity_discrete}
 There exists positive constant $\beta_{\mathcal{F}}$ such that for all $w,v \in H^{\frac32+\delta}(\Omega^-)$, for $\delta > 0$, and $\widetilde{w}, \widetilde{v} \in L^2(\Gamma)$
 \begin{equation}
  \label{eq:continuity_discrete_fem}
  \left|a_h\left(\left(w, \widetilde{w} \right), \left(v, \widetilde{v}\right)\right)\right| \leq \beta_{\mathcal{F}} \left\|(w,\widetilde{w}) \right\|_{\mathcal{F}} \left\|(v,\widetilde{v}) \right\|_{\mathcal{F}}.
 \end{equation}
 There exists  positive constant $\beta_{\mathcal{B}}$ such that for all $w,v \in H^{\frac{1}{2}}(\Gamma)$, $\lambda, \varphi \in L^{2}(\Gamma)$ and $\widetilde{w}, \widetilde{v} \in L^{2}(\Gamma)$
 \begin{equation}
  \label{eq:continuity_discrete_bem}
  \left|b_h\left( \left(w, \lambda, \widetilde{w}\right), \left(v, \varphi, \widetilde{v}\right)\right)\right| \leq \beta_{\mathcal{B}} \left\|(w,\lambda, \widetilde{w}) \right\|_{\mathcal{B}} \left\|(v,\varphi, \widetilde{v}) \right\|_{\mathcal{B}}.
 \end{equation}
\end{lemma}
\begin{proof}
 We use Cauchy-Schwarz inequality to obtain~\eqref{eq:continuity_discrete_fem}. In the case of equation~\eqref{eq:continuity_discrete_bem}, we use Cauchy-Schwarz inequality and Lemma~\ref{l:continuity_continuous}.
\end{proof}
To show the well-posedness of~\eqref{eq:fem_bem_disctrete} we need the ellipticity of the bilinear forms $a_h$ and $b_h$
\begin{lemma}[Coercivity]
  \label{l:coercivity_discrete}
 Assume that  positive constant $\tau$ is large enough. 
 Then, there exists  positive constant $\alpha$ such that for all $\left({w_h}, \lambda_h, {\widetilde{w}_h}\right) \in \mathcal{V}_{h} \times \Lambda_h^l \times M_{h}^{m}$
 \begin{align}
  \label{eq:coercivity_discrete}
  a_h\left(\left(w_h^-, \widetilde{w}_h \right), \left(w_h^-, \widetilde{w}_h \right)\right) 
   + & b_h \left( \left(w_h^+, \lambda_h, {\widetilde{w}_h}\right), \left({w_h}^+, \lambda_h, {\widetilde{w}_h}\right)\right) \\ \nonumber & \geq \alpha \left( \left\|(w_h,\widetilde{w}_h) \right\|_{\mathcal{F}}^2 + \left\|\left({w_h}, \lambda_h, {\widetilde{w}_h}\right) \right\|_{\mathcal{B}}^2 \right).
 \end{align}
%  Then, there exists  positive constant $\alpha_{\mathcal{F}}$ such that for all $\left({w_h}, {\widetilde{w}_h}\right) \in V_{h}^{j} \times M_{h}^{m}$
%  \begin{equation}
%   \label{eq:coercivity_discrete_fem}
%   a_h\left(\left(w_h, \widetilde{w}_h \right), \left(w_h, \widetilde{w}_h \right)\right) \geq \alpha_{\mathcal{F}} \left\|(w_h,\widetilde{w}_h) \right\|_{\mathcal{F}}^2
%  \end{equation}
%  There exists  positive constant $\alpha_{\mathcal{B}}$ such that for all $\left({w_h}, \lambda_h {\widetilde{w}_h}\right) \in W_h^k \times \Lambda_h^l \times M_{h}^{m}$
%  \begin{equation}
%   \label{eq:coercivity_discrete_bem}
%   b_h\left( \left({w_h}, \lambda_h {\widetilde{w}_h}\right), \left({w_h}, \lambda_h {\widetilde{w}_h}\right)\right) \geq \alpha_{\mathcal{B}} \left\|\left({w_h}, \lambda_h {\widetilde{w}_h}\right) \right\|_{\mathcal{B}}^2
%  \end{equation}
\end{lemma}
\begin{proof}
 Let us start with bilinear form $a_h$. First we assume that $\varepsilon\ge \varepsilon_{min}>0$
 \begin{align*}
  a_h\left(\left(w_h^-, \widetilde{w}_h \right), \left(w_h^-, \widetilde{w}_h \right)\right) = & |w_h^-|^2_{H^1(\Omega^-)} + %_{\min} 
  \|\varepsilon^{1/2} w_h^-\|_{\Omega^-}^2 - 2 \left<\partial_n w_h^-, w_h^- - \widetilde{w}_h\right>_{\Gamma} \\
  &  + \tfrac{\tau}{h} \left\| w_h^- - \widetilde{w}_h\right\|_{L^2(\Gamma)}^2.
 \end{align*}
 Using Cauchy-Schwarz and trace inequalities~\eqref{eq:discrete_trace_inequality}, followed by Young's inequality, we arrive at
 \begin{align*}
   a_h\left(\left(w_h^-, \widetilde{w}_h \right), \left(w_h^-, \widetilde{w}_h \right)\right) \geq & |w_h^-|^2_{H^1(\Omega^-)} + \varepsilon_{\min}  \|w_h^-\|_{\Omega^-}^2 + \tfrac{\tau}{h} \left\| w_h^- - \widetilde{w}_h\right\|_{L^2(\Gamma)}^2 \\
   & - 2 \left\|\partial_n w_h^- \right\|_{L^2(\Gamma)} \left\|w_h^- - \widetilde{w}_h\right\|_{L^2(\Gamma)} \\
   \geq& |w_h^-|^2_{H^1(\Omega^-)} + \varepsilon_{\min}  \|w_h^-\|_{\Omega^-}^2  + \tfrac{\tau}{h} \left\| w_h^- - \widetilde{w}_h\right\|_{L^2(\Gamma)}^2  \\
   & - 2  \left| w_h^- \right|_{H^1(\Omega^-)} \left( C_{max}h^{-\frac{1}{2}} \left\|w_h^- - \widetilde{w}_h\right\|_{L^2(\Gamma)} \right) \\
   \geq & \tfrac{1}{2} |w_h^-|^2_{H^1(\Omega^-)} + \varepsilon_{\min}   \|w_h^-\|_{\Omega^-}^2  + \tfrac{\tau - 2 C_{max}^2}{h} \left\| w_h^- - \widetilde{w}_h\right\|_{L^2(\Gamma)}^2.
 \end{align*}
We finish by applying the equivalence of the norms~\eqref{eq:fem_norm_equivalence} under the assumption  that $\tau > 2 C_{max}^2$.

In the case of bilinear form $b_h$, by using the results from Lemma~\ref{l:coercivity_continuous} we obtain, with $\bar w_h = |\Gamma|^{-1} \int_\Gamma w^+_h ~\mbox{d}s$,
 \begin{align*}
  b_h\left( \left(w_h^+, \lambda_h, {\widetilde{w}_h}\right), \left({w_h}^+, \lambda_h, {\widetilde{w}_h}\right)\right) \geq & \alpha_V \left\|\lambda_h \right\|_{ H^{-\frac{1}{2}}(\Gamma)}^2 + \alpha_W \left\|w_h^+ - \bar w_h\right\|_{ H^{\frac{1}{2}}(\Gamma)}^2\\
  & + \tfrac{\tau}{h} \left\| w_h^+ - \widetilde{w}_h\right\|_{L^2(\Gamma)}^2.
 \end{align*}
 Observe that when $\varepsilon>0$ we may bound
 \begin{align*}
\|w_h^+\|_{ H^{\frac{1}{2}}(\Gamma)} &\leq \|w_h^+ - \bar w_h\|_{ H^{\frac{1}{2}}(\Gamma)} + \|\bar w_h\|_{ L^2(\Gamma)} \\
&\leq \|w_h^+ - \bar w_h\|_{ H^{\frac{1}{2}}(\Gamma)} + \|w_h^+ - \widetilde w_h\|_{ L^2(\Gamma)} + \|w_h^- - \widetilde w_h\|_{ L^2(\Gamma)}+\|w_h^-\|_{ L^2(\Gamma)}\\
&\leq \|w_h^+ - \bar w_h\|_{ H^{\frac{1}{2}}(\Gamma)} + \|w_h^+ - \widetilde w_h\|_{ L^2(\Gamma)} + \|w_h^- - \widetilde w_h\|_{L^2(\Gamma)} + C \|w_h^-\|_{H^1(\Omega^-)}
 \end{align*}
 where we applied the trace inequality \eqref{eq:trace_theorem} in the last step. The right hand side is controlled by the lower bounds on $a_h$ and $b_h$ above. 
 Once again, we finish by applying the equivalence of the norms~\eqref{eq:bem_norm_equivalence}.

In case $\varepsilon=0$ we need to show that a Poincar\'e inequality holds, similar to \eqref{eq:Pcare}, this time on the form
\begin{multline}
c_P \|w_h^-\|_{H^1(\Omega^-)} \leq |w_h^-|_{H^1(\Omega^-)} + \tfrac{1}{h} \left\| w_h^- - \widetilde{w}_h\right\|_{L^2(\Gamma)}+\tfrac{1}{h} \left\| w_h^+ - \widetilde{w}_h\right\|_{L^2(\Gamma)}\\ + \left\|\lambda_h \right\|_{ H^{-\frac{1}{2}}(\Gamma)}  + \left\|w_h^+ - \bar w_h\right\|_{ H^{\frac{1}{2}}(\Gamma)}
\end{multline}
 To this end, since coercivity holds up to a constant, we may assume that $w^-_h = \tilde w_h = w^+_h = \bar w_h$ and proceed verbatim as in the continuous case, since in that case the continuous and discrete expressions corresponding to \eqref{eq:sing_pot} are the same.
\end{proof}

The existence and uniqueness of the solution of problem~\eqref{eq:fem_bem_disctrete} is achieved by using the Lax-Milgram theorem. In addition, the  proposed method is consistent as the following result shows.

\begin{lemma}[Consistency]
 \label{l:consistency}
 Let $\delta > 0$, $u^- \in H^{\frac32+\delta}(\Omega^-)$, $u^+ \in H^{\frac{1}{2}}(\Gamma)$ and $\partial_n u^+ = \lambda \in L^{2}(\Gamma)$ be the solution of problem~\eqref{eq:divided_problem} and $\widetilde{u} = u^- = u^+$ on $\Gamma$. If $\left(u_h, \lambda_h, \widetilde{u}_h\right) \in \mathcal{V}_h\times \Lambda_{h}^{l} \times M_{h}^{m}$ solves~\eqref{eq:fem_bem_disctrete} then, for all $\left(v_h, \varphi_h, \widetilde{v}_h\right) \in \mathcal{V}_h \times \Lambda_{h}^{l} \times M_{h}^{m}$ the following holds
 \begin{equation*}
  a_h\left(\left(u^--u_h^-, \widetilde{u} - \widetilde{u}_h \right), \left(v_h^-, \widetilde{v}_h\right)\right) + b_h\left( \left(u^+-u_h^+, \lambda - \lambda_h, \widetilde{u} -\widetilde{u}_h\right), \left(v_h^+, \varphi_h, \widetilde{v}_h\right)\right)  =0.
 \end{equation*}
\end{lemma}
\begin{proof}
Because of the transmission conditions from~\eqref{eq:divided_problem} we have $u = u^+ = u^-$ and $\partial_n u = \partial_n u^+ = \partial_n u^-$ on $\Gamma$
  \begin{align*}
  a_h\left(\left(u-u_h^-, \widetilde{u} - \widetilde{u}_h \right), \left(v_h^-, \widetilde{v}_h\right)\right) = & \left<\partial_n u, \widetilde{v}_h\right>_{\Gamma} - \left<u - \widetilde{u}, \partial_n v_h^- \right>_{\Gamma}  \\
  & + \tfrac{\tau}{h} \left<u - \widetilde{u}, v_h^- - \widetilde{v}_h\right>_{\Gamma} ,\\
  b_h\left( \left(u-u_h^+, \lambda - \lambda_h, \widetilde{u} -\widetilde{u}_h\right), \left(v_h^+, \varphi_h, \widetilde{v}_h\right)\right)  = & -  \left< \lambda, \widetilde{v}_h\right>_{\Gamma} -  \left<u - \widetilde{u}, \varphi \right>_{\Gamma} \\
  &+  \tfrac{\tau}{h} \left<u - \widetilde{u}, v_h^+ - \widetilde{v}_h\right>_{\Gamma}.
 \end{align*}
 By adding above expressions and using the facts that $\lambda = \partial_n u$ and $\widetilde{u} = u$ on $\Gamma$, we obtain consistency.
\end{proof}

%%%%%%%%%%%%%%%%%%%%%%%%%%%%%%%%%%%%%%%%%%%%%%%%%%%%%%%%%%%%%%%%%%%%%%%%%%%%%%%%%%%%%%%%%%%%%%%%%%%%%%%%%%%%
% Error analysis for discrete problem
%%%%%%%%%%%%%%%%%%%%%%%%%%%%%%%%%%%%%%%%%%%%%%%%%%%%%%%%%%%%%%%%%%%%%%%%%%%%%%%%%%%%%%%%%%%%%%%%%%%%%%%%%%%%

\subsection{Error analysis}
\label{sec:error}

In this section we present the error estimates for the method. These estimates are proved using the following norm
\begin{equation}
	\big\|\big(u, \lambda, \widetilde{u}\big)\big\|_h := \left\|\left(u^-, \widetilde{u}\right)\right\|_{\mathcal{F}} + \left\|\left(u^+, \lambda, \widetilde{u} \right)\right\|_{\mathcal{B}}.
\end{equation}
\noindent The first step is the following version of Cea's lemma.

%%%%%%%%%%%%%%%%%%%%%%%%%%%%%%%%%%%%%%%%%%%%%%%%%%
%% Cea's Lemma
\begin{lemma}[Cea's Lemma]
\label{l:Cea}
Let $\delta > 0$, $u^- \in H^{\frac32+\delta}(\Omega^-)$, $u^+ \in H^{\frac{1}{2}}(\Gamma)$ and $\partial_n u^+ = \lambda \in L^{2}(\Gamma)$ be the solution of problem~\eqref{eq:divided_problem}, $\widetilde{u} = u^- = u^+$ on $\Gamma$, and let $\left(u_h, \lambda_h, \widetilde{u}_h\right) \in \mathcal{V}_h\times \Lambda_{h}^{l} \times M_{h}^{m}$ solve~\eqref{eq:fem_bem_disctrete}. Then there exists $C > 0$, independent of $h$, such that
\begin{equation}
\label{eq:cea}
\left\|\left(u-u_h, \lambda - \lambda_h, \widetilde{u}- {\widetilde{u}_h}\right)\right\|_h \leq C \inf_{\left(v_h, \varphi_h, \widetilde{v}_h\right) \in \mathcal{V}_h \times \Lambda_{h}^{l} \times M_{h}^{m}} \left\|\left(u-v_h, \lambda - \varphi_h, \widetilde{u}- {\widetilde{v}_h}\right)\right\|_h.
\end{equation}
\end{lemma}
%%%%%%%%%%%%%%%%%%%%%%%%%%%%%%%%%%%%%%%%%%%%%%%%%%
%% Cea's proof
\begin{proof}
Let us denote
\begin{align*}
A_h \left(\left(w_h, \phi_h, \widetilde{w}_h\right), \left(v_h, \varphi_h, \widetilde{v}_h\right)\right) := & a \left(\left(w_h^-, \widetilde{w}_h\right),\left(v_h^-, \widetilde{v}_h\right)\right) \\
& + b\left(\left(w_h^+, \phi_h, \widetilde{w}_h\right), \left(v_h^+, \varphi_h, \widetilde{v}_h\right)\right).
\end{align*}
%%%%%%%%%%%%%%%%%%%%%%%%%%%%%%%%%%%%%%%%%%%%%%%%%%
%% Inf-sup condition
Using Lemma~\ref{l:coercivity_discrete}, we get that there exists $\alpha > 0$, independent of $h$, such that for all $\left(v_h, \varphi_h, \widetilde{v}_h\right) \in \mathcal{V}_h \times \Lambda_{h}^{l} \times M_{h}^{m}$ there exists $\left(w_h, \phi_h, \widetilde{w}_h\right) \in \mathcal{V}_h\times \Lambda_{h}^{l} \times M_{h}^{m}$ such that $\left\|\left(w_h, \phi_h, \widetilde{w}_h\right)\right\|_h=1$, and 
\begin{equation}
\label{eq:coercivity_l}
A_h \left(\left(v_h, \varphi_h, \widetilde{v}_h\right), \left(w_h, \phi_h, \widetilde{w}_h\right)\right) \geq \alpha \left\|\left(v_h, \varphi_h, \widetilde{v}_h\right)\right\|_h.
\end{equation}
%%%%%%%%%%%%%%%%%%%%%%%%%%%%%%%%%%%%%%%%%%%%%%%%%%
%% Boundedness
Now using Lemma~\ref{l:continuity_discrete}, we get continuity of $A_h$, there exists $\beta > 0$
\begin{align}
\label{eq:continuity_l}
\left|A_h \left(\left(v, \varphi, \widetilde{v}\right)\right), \left(w, \phi, \widetilde{w}\right)\right| & \leq 
\beta \left\|\left(v, \varphi, \widetilde{v}\right)\right\|_h \left\|\left(w, \phi, \widetilde{w}\right)\right\|_h.
\end{align}
%%%%%%%%%%%%%%%%%%%%%%%%%%%%%%%%%%%%%%%%%%%%%%%%%%
%% Triangle inequality
Let $\left(v_h, \varphi_h, \widetilde{v}_h\right) \in \mathcal{V}_h \times \Lambda_{h}^{l} \times M_{h}^{m}$. 
Then, using the triangle inequality we see that
\begin{align*}
\left\|\left(u-u_h, \lambda - \lambda_h, \widetilde{u}- {\widetilde{u}_h}\right)\right\|_h \leq & \left\|\left(u-v_h, \lambda - \varphi_h, \widetilde{u}- {\widetilde{v}_h}\right)\right\|_h \\
& + \left\|\left(v_h-u_h, \varphi_h-\lambda_h, \widetilde{v}_h -\widetilde{u}_h\right)\right\|_h.
\end{align*}
%%%%%%%%%%%%%%%%%%%%%%%%%%%%%%%%%%%%%%%%%%%%%%%%%%
%% Second term
and it follows from \eqref{eq:coercivity_l}, Lemma~\ref{l:consistency} and~\eqref{eq:continuity_l} that
\begin{align*}
\left\|\left(v_h-u_h, \varphi_h-\lambda_h, \widetilde{v}_h -\widetilde{u}_h\right)\right\|_h \leq &  \tfrac{1}{\alpha} A_h \left(\left(v_h-u, \varphi_h- \lambda, {\widetilde{v}_h} - \widetilde{u} \right), \left(w_h, \phi_h, \widetilde{w}_h\right)\right) \\
& +  \tfrac{1}{\alpha} A_h \left(\left(u - u_h, \lambda - \lambda_h, \widetilde{u} - \widetilde{u}_h \right), \left(w_h, \phi_h, \widetilde{w}_h\right)\right) \\
\leq &  \tfrac{\beta}{\alpha} \left\|\left(v_h-u, \varphi_h- \lambda, {\widetilde{v}_h} - \widetilde{u} \right)\right\|_h.
\end{align*}
Thus, we get~\eqref{eq:cea} with $C := 1+\tfrac{\beta}{\alpha}$.
\end{proof}
% \textcolor{red}{
% \begin{remark}The norms are still dodgy. We can not allow global $H^s(\Gamma)$ norms of $u$ for $s>1$ on polyhedral domains. I have tried to save this by modifying your remark, but it is still strange, since the remark comes after the theorem... please check this carefully referring to Steinbach's book and also the Sauter/Schwab book.
% \end{remark}}
\begin{lemma}[Energy norm estimates]
\label{l:energy_norm}
For $s > \tfrac32$, $r > 1$ and $p \geq \tfrac12$, let $u^- \in H^{s}(\Omega^-)$, $u^+ \in \widetilde H^{r}(\Gamma)$ and $\partial_n u^+ = \lambda \in \widetilde H^{p}(\Gamma)$ be the solution of problem~\eqref{eq:divided_problem}. On $\Gamma$ there holds  $u^-  = u^+ = \widetilde{u} $ on $\Gamma$. Let $\left(u_h, \lambda_h, \widetilde{u}_h\right) \in \mathcal{V}_h\times \Lambda_{h}^{l} \times M_{h}^{m}$ solve~\eqref{eq:fem_bem_disctrete}. If the mesh is quasi uniform, then there exists $C > 0$, independent of $h$, such that
\begin{multline}
\label{eq:energy_norm}
\left\|\left(u-u_h, \lambda - \lambda_h, \widetilde{u}- {\widetilde{u}_h}\right)\right\|_h \leq \\ C \left( h^{\sigma-1} \|u\|_{H^{\sigma}(\Omega^-)} +  h^{\phi-\frac{1}{2}} \|u\|_{\widetilde H^{\phi}(\Gamma)} + h^{\xi-\frac{1}{2}} \|u\|_{\widetilde H^{\xi}(\Gamma)}  + h^{\psi+\frac{1}{2}} \|\lambda\|_{\widetilde H^{\psi}(\Gamma)} \right),
\end{multline}
where $\sigma = \min \{j+1, s\}$, $\phi = \min\{k+1, r\}$, $\xi = \min \{m+1, r\}$ and $\psi = \min \{l+1, p\}$.
\end{lemma}

\begin{proof}
The result is a consequence of \eqref{eq:cea} and approximation.
Applying triangle inequality and trace inequality~\cite[Theorem 1.6.6]{MR2373954} followed by Young's inequality, we obtain
 \begin{align*}
  \left\|\left(u - v_h, \widetilde{u} - \widetilde{v}_h\right)\right\|_{\mathcal{F}}^2 :=& \left\|u - v_h\right\|_{H^1(\Omega^-)}^2 + \tfrac{\tau}{h}   \left\|u - v_h -\left(\widetilde{u} - \widetilde{v}_h\right)\right\|_{L^2(\Gamma)}^2  \\
  &   + h  \left\|\partial_n u - \partial_n v_h\right\|_{L^2(\Gamma)}^2 \\
  \leq &  \left\|u - v_h\right\|_{H^1(\Omega^-)}^2 + C_1 \left( \tfrac{\tau}{h^2} \left\|u - v_h\right\|_{L^2(\Omega^-)}^2 +  \tau  \sum_{K \in \mathcal{T}_h}\left\|u - v_h\right\|_{H^1(K)}^2\right)   \\
  & + C_2 \left( \sum_{K \in \mathcal{T}_h}\left\|u - v_h\right\|_{H^1(K)}^2 + h^2 \sum_{K \in \mathcal{T}_h} \left\|u - v_h\right\|_{H^2(K)}^2\right) + \tfrac{\tau}{h} \left\|\widetilde{u} - \widetilde{v}_h\right\|_{L^2(\Gamma)}^2.
 \end{align*}
Using approximation results~\cite[Theorem 1.109]{MR2050138} % on $\Omega$ 
and~\cite[Theorem 4.3.19]{MR2743235} for the last term, we claim
\begin{equation*}
 \inf_{\left(v_h, \widetilde{v}_h\right) \in \mathcal{V}_h \times M_{h}^{m}} \left\|\left(u - v_h, \widetilde{u} - \widetilde{v}_h\right)\right\|_{\mathcal{F}} \leq C_{\mathcal{F}} \left(h^{\sigma-1} \|u\|_{H^{\sigma}(\Omega^-)} + h^{\xi-\frac{1}{2}} \|u\|_{\widetilde H^{\xi}(\Gamma)} \right).
\end{equation*}
% where $\sigma = \min \{k+1, m\}$.
% 
For the boundary part, by applying triangle inequality, we obtain
 \begin{align*}
 \left\|\left(u - v_h, \lambda - \lambda_h, \widetilde{u} - \widetilde{v}_h\right)\right\|_{\mathcal{B}}^2 := & \|u - v_h\|_{H^{\frac{1}{2}}(\Gamma)}^2 + \|\lambda - \lambda_h\|_{H^{-\frac{1}{2}}(\Gamma)}^2  + h \left\|\lambda - \lambda_h\right\|_{L^2(\Gamma)}^2 \\
 & + \tfrac{\tau}{h} \left\|u - v_h - \left(\widetilde{u} - \widetilde{v}_h\right)\right\|_{L^2(\Gamma)}^2 \\
 \leq & \|u - v_h\|_{H^{\frac{1}{2}}(\Gamma)}^2 + \|\lambda - \lambda_h\|_{H^{-\frac{1}{2}}(\Gamma)}^2 + h \left\|\lambda - \lambda_h\right\|_{L^2(\Gamma)}^2\\
 &  + \tfrac{\tau}{h} \left\|u - v_h\right\|_{L^2(\Gamma)}^2  + \tfrac{\tau}{h} \left\|\widetilde{u} - \widetilde{v}_h\right\|_{L^2(\Gamma)}^2.
 \end{align*}
 Using approximation results~\cite[Theorems 4.3.19, 4.3.20 and 4.3.22]{MR2743235}% and \cite[Theorems 10.4]{MR2361676} when $0 \leq p < \tfrac12$
 , we claim
\begin{multline*}
 \inf_{\left(v_h, \lambda_h, \widetilde{v}_h\right) \in \mathcal{V}_h \times \Lambda_h^l \times M_{h}^{m}} \left\|\left(u - v_h, \lambda - \lambda_h, \widetilde{u} - \widetilde{v}_h\right)\right\|_{\mathcal{B}} \leq \\ C_{\mathcal{B}} \left( h^{\phi-\frac{1}{2}} \|u\|_{\widetilde H^{\phi}(\Gamma)} + h^{\xi-\frac{1}{2}} \|u\|_{\widetilde H^{\xi}(\Gamma)}  + h^{\psi+\frac{1}{2}} \|\lambda\|_{\widetilde H^{\psi}(\Gamma)}\right).
\end{multline*}
% where $\phi = \min\{k+1, s\}$ and $\psi = \min \{l+1, p\}$.
% 
We conclude the proof by applying Lemma~\ref{l:Cea}.
\end{proof}

%%%%%%%%%%%%%%%%%%%%%%%%%%%%%%%%%%%%%%%%%%
% Iterative solution
%%%%%%%%%%%%%%%%%%%%%%%%%%%%%%%%%%%%%%%%%%

\section{Iterative solution}
For the solution of the linear system we will iterate on the Schur complement for the trace variable, solving independently in the two sub domains. To justify this split approach we here show that a simple relaxed Jacobi iteration on the two systems will converge. The condition number of the Schur complement can be analysed using the arguments of \cite[Section 4]{MR4021278}.
\label{sec:iterative}
\begin{enumerate}
    \item Given $\widetilde u^n$ solve for $u^{n+1}$ and $\lambda^{n+1}$ by solving the linear system
    \[
    A_h[(u^{n+1}, \lambda^{n+1}, \widetilde u^n), (v,\varphi,0)] = \int_{\Omega^{-}} f v ~dx.
    \]
    \item Given $u^{n+1}$ and $\lambda^{n+1}$, solve for the new trace variable $\widetilde u^{n+1}$, for $\sigma>0$,
    \[
    A_h[(u^{n+1}, \lambda^{n+1}, \widetilde u^{n+1}), (0,0,\widetilde v)] + \sigma \tfrac{\tau}{h} \left<\widetilde u^{n+1} - \widetilde u^{n},\widetilde v  \right>_{\Gamma} = 0.
    \]
\end{enumerate}
To prove that the iterative algorithm converges we only need to show that if $f = 0$, $u^{n+1}$, $\widetilde u^{n+1}$ and $\lambda^{n+1}$ all go to zero as $n \rightarrow \infty$.

We add and subtract $\widetilde u^{n+1}$ in the first equation and add the second to obtain
 \begin{multline*}
    A_h[(u^{n+1}, \lambda^{n+1}, \widetilde u^{n+1}), (v,\varphi,\widetilde v)] + \sigma \tfrac{\tau}{h} \left<\widetilde u^{n+1} - \widetilde u^{n},\widetilde v  \right>_{\Gamma} = \\
    \left<(\widetilde u^{n+1} - \widetilde u^{n}), \partial_n v + \varphi - \tfrac{\tau}{h} ( (v^- + v^+ )\right>_{\Gamma}.
    \end{multline*}
    Test this equation with $u^{n+1}, \ \lambda^{n+1}, \ \widetilde u^{n+1}$ and use coercivity to obtain
    \begin{multline*}
    \tfrac12 \sigma \tfrac{\tau}{h} \|\widetilde u^{N}\|_{L^2(\Gamma)}^2 +  \sum_{n=0}^{N-1} ( \alpha \|(u^{n+1}, \lambda^{n+1}, \widetilde u^{n+1})\|_h^2 + \tfrac12 \sigma \tfrac{\tau}{h} \|\widetilde u^{n+1} - \widetilde u^n\|_{L^2(\Gamma)}^2)\\
    \leq \tfrac12 \sigma \tfrac{\tau}{h} \|\widetilde u^{0}\|_{L^2(\Gamma)}^2 + \sum_{n=0}^{N-1}  \left<(\widetilde u^{n+1} - \widetilde u^n), \partial_n u^{n+1} + \lambda^{n+1} - \tfrac{\tau}{h} ( (u^{n+1})^- + (u^{n+1})^+ )\right>_{\Gamma}.
    \end{multline*}
    Here we used the well known formula
    \begin{equation}\label{eq:telescope}
\sum_{n=0}^{N-1} \left<(\widetilde u^{n+1} - \widetilde u^n), \widetilde u^{n+1}\right>_{\Gamma} = \tfrac12  \|\widetilde u^{N}\|_{L^2(\Gamma)}^2 - \tfrac12  \|\widetilde u^{0}\|_{L^2(\Gamma)}^2 + \tfrac12  \sum_{n=0}^{N-1} \|\widetilde u^{n+1} - \widetilde u^n\|_{L^2(\Gamma)}^2 .
\end{equation}
}
Considering the terms on the right hand side and using trace inequality~\eqref{eq:discrete_trace_inequality} we see that
\[
 \left<(\widetilde u^{n+1} - \widetilde u^n), \partial_n (u^{n+1})^-\right>_{\Gamma} \leq \tfrac14 \sigma \tfrac{\tau}{h} \|\widetilde u^{n+1} - \widetilde u^n\|^2_{L^2(\Gamma)} + C_{max} \tau^{-1} \sigma^{-1}\|u^{n+1}\|_{H^1(\Omega^-)}^2.
\]
Using the duality pairing between $H^{\frac12}$ and $H^{-\frac12}$ followed by the global inverse inequality $\|\widetilde u^{n+1} - \widetilde u^n\|_{H^{\frac12}(\Gamma)} \leq C_t h^{-\frac12} \|\widetilde u^{n+1} - \widetilde u^n\|_{L^2(\Gamma)}$ (see~\cite[Theorem 4.4.3]{MR2743235}) and Young's inequality we have
\[
 \left<(\widetilde u^{n+1} - \widetilde u^n), \lambda^{n+1} \right>_{\Gamma} \leq \tfrac14 \sigma \tfrac{\tau}{h} \|\widetilde u^{n+1} - \widetilde u^n\|^2_{L^2(\Gamma)} + C_t^{-2} \tau^{-1} \sigma^{-1}\|\lambda^{n+1}\|_{H^{-\frac12}(\Gamma)}^2.
\]
Finally 
\begin{align*}
\left<(\widetilde u^{n+1} - \widetilde u^n), \tfrac{\tau}{h} ( (u^{n+1})^- + (u^{n+1})^+ )\right>_{\Gamma} \leq & \ \tfrac14 \sigma \tfrac{\tau}{h} \|\widetilde u^{n+1} - \widetilde u^n\|^2_{L^2(\Gamma)}& \\
+ \sigma^{-1} %\tau/h
\big( \tfrac{\tau}{h} \|(u^{n+1})^+ - \widetilde u^{n+1}\|^2_{L^2(\Gamma)} &+\tfrac{\tau}{h} \|(u^{n+1})^- - \widetilde u^{n+1}\|^2_{L^2(\Gamma)} \big) \\
&+ 2 \tfrac{\tau}{h} \left<(\widetilde u^{n+1} - \widetilde u^n), \widetilde u^{n+1}\right>_{\Gamma}.
\end{align*}
Using the once again the telescoping property \eqref{eq:telescope} we see that
for %$\sigma$ sufficiently large 
$$\sigma > \tau^{-1} \alpha^{-1} \max\left\{C_{max}, C_t^{-2}, \tau\right\} + 2,$$ %(give precise value of $\sigma$ in terms of $\tau$ and the constants of trace inequalities and inverse inequalities)
the right hand sides can all be absorbed in the left hand side to yield
\[
 \sum_{n=0}^{N-1} \left( \widetilde \alpha \|(u^{n+1}, \lambda^{n+1}, \widetilde u^{n+1})\|_h^2 + \tfrac{1}{4} \sigma \tfrac{\tau}{h} \|\widetilde u^{n+1} - \widetilde u^n\|_{L^2(\Gamma)}^2\right) \leq \tfrac12 (\sigma  - 2) \tfrac{\tau}{h} \|\widetilde u^{0}\|_{L^2(\Gamma)}^2.
\]
It follows that as $N \rightarrow \infty$ $u^{n+1}$, $\widetilde u^{n+1}$ and $\lambda^{n+1}$ all go to zero, since the sum of the left hand side has to be bounded by the constant of the right hand side.

%%%%%%%%%%%%%%%%%%%%%%%%%%%%%%%%%%%%%%%%%%
% Numerical experiments
%%%%%%%%%%%%%%%%%%%%%%%%%%%%%%%%%%%%%%%%%%

\section{Numerical experiments}
\label{sec:numerics}

In our experiment tests we consider $\varepsilon = 1$ and $\mathcal{V}_h\times \Lambda_{h}^{l} \times M_{h}^{m}$ with $j = k = m = 1$. The value $l$ varies depending on the geometry of domains considered. We let the trace meshes $\mathcal{G}_1$ and $\mathcal{G}_2$ coincide with the trace mesh of $\mathcal{T}_h$ on $\Gamma$. For our experiments we use two numerical softwares: FEniCS~\cite{AlnaesBlechta2015a} and Bempp~\cite{SmigajEtAl2015}.
%We consider two subdomains $\Omega^-, \Omega^+$ and on each of them we are imposing the following Dirichlet data at the continuous level $u^- =  \widetilde{u}$ and $u^+ =  \widetilde{u}$ on $\Gamma$. 
We use the solution of interior and exterior Dirichlet boundary value problems to construct a Schur complement system
solving the following equations
\begin{equation*}
   \tfrac{\tau}{h} \left( u_s - \widetilde{u}\right) + \tfrac{\tau}{h} \left( u_m - \widetilde{u}\right) - \left(\partial_n u -\lambda\right)  = 0 \mbox{ on } \Gamma.
\end{equation*}
The solution $\widetilde{u}$ on $\Gamma$ of the Schur complement is obtained using the nested conjugate gradient method (CG)~\cite{MR0060307}. Although one can use direct solvers to solve the interior and exterior Dirichlet boundary value problems, we here use iterative solvers to apply preconditioners.  The interior Dirichlet boundary value problem that is a symmetric system associated with bilinear form $a_h$~\eqref{eq:fem_bem_disctrete} is solved by using FEniCS and CG without and with algebraic multigrid preconditioner. The discrete exterior problem associated with bilinear form $b_h$~\eqref{eq:fem_bem_disctrete} is not symmetric, however as we shown in Section~\ref{sec:symmetric}, we can apply the Steklov-Poincar\'e operator to the flux variable and transform the equations into a symmetric system. For clarity of the code, we here simply used the generalized minimal residual method (GMRES)~\cite{MR848568} without or with mass matrix preconditioner to solve in Bempp the external Dirichlet boundary value problem.
The tolerance of the iterative solvers is chosen to be not greater the $10^{-8}$. 
A Jupyter notebook demonstrating the functionality used in this paper will be made available at \url{www.bempp.com}.
 
\subsection{Choice of parameter $\tau$}

Thanks to Lemma~\ref{l:coercivity_discrete} we know that the stabilisation parameter $\tau$ in the discrete problem~\eqref{eq:fem_bem_disctrete} must be large enough to assure the coercivity. We start with an experiment showing how the value of the parameter $\tau$ influences the convergence and number of iterations. We consider $\Omega^-$ as a unit sphere with boundary $\Gamma$. We define 
\begin{align*}
 u^- (x, y, z) & = \frac{1}{2 \pi} \sin\left(\pi (x^2+y^2+z^2)\right)  + \frac{1}{2 \pi} \cos\left( \pi(x^2+y^2+z^2)\right) + \frac{2 \pi+1}{2 \pi}, \\
 u^+ (x, y, z) & = \frac{1}{\sqrt{x^2+y^2+z^2}}.
\end{align*}
It is easy to check that for the unit sphere domain $\Omega^-$ the above elementary functions are the solution of our problem~\eqref{eq:divided_problem}.

\begin{figure}[!h] 
\centering
 \subfloat[The error of the interior solution.
\label{fig:tau_FEM_error}]{\resizebox {0.32\columnwidth} {!}{
    \includegraphics{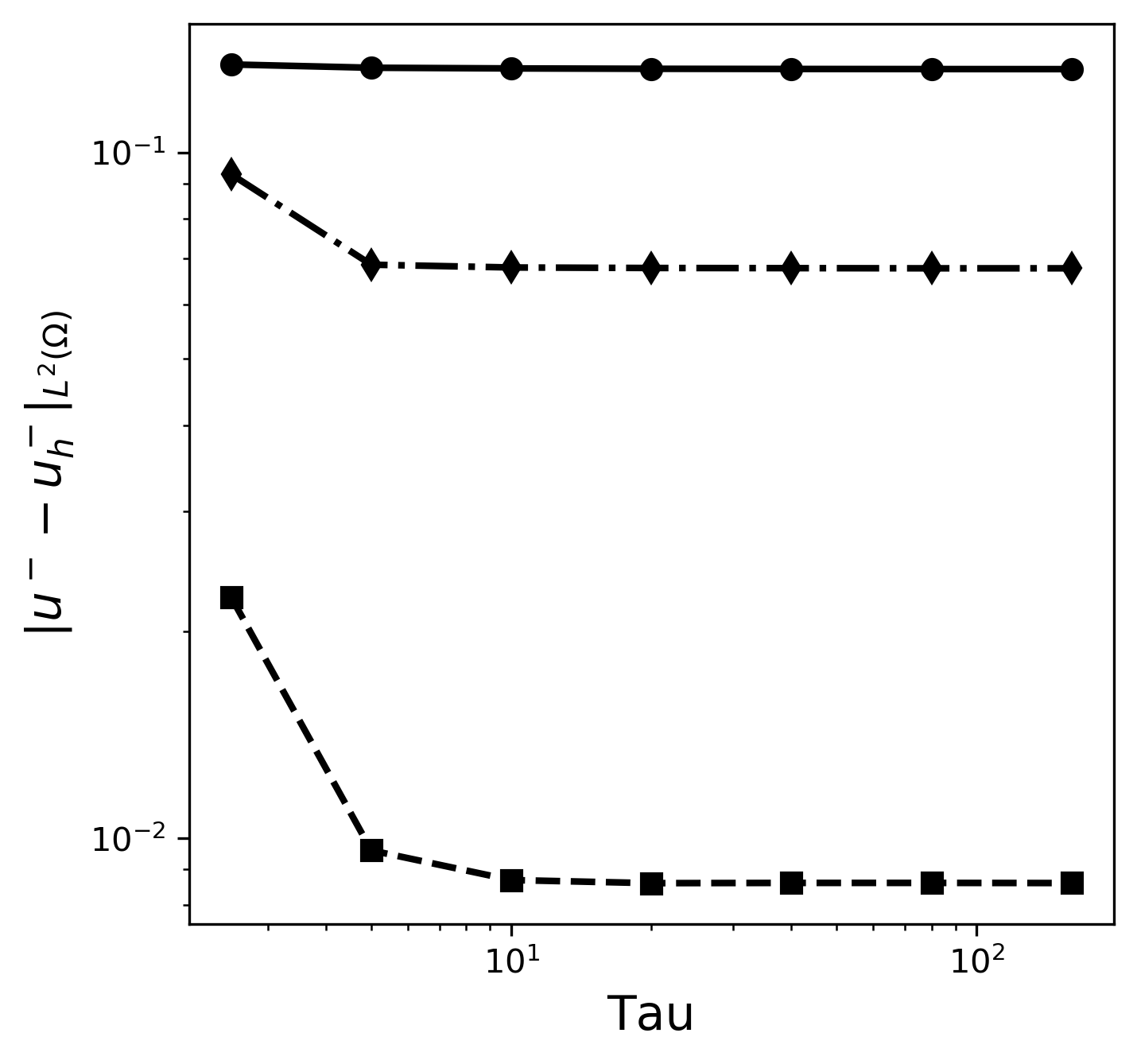}
}} \hfill
 \subfloat[The error of the exterior solution.
\label{fig:tau_BEM_error}]{\resizebox {0.32\columnwidth} {!}{
    \includegraphics{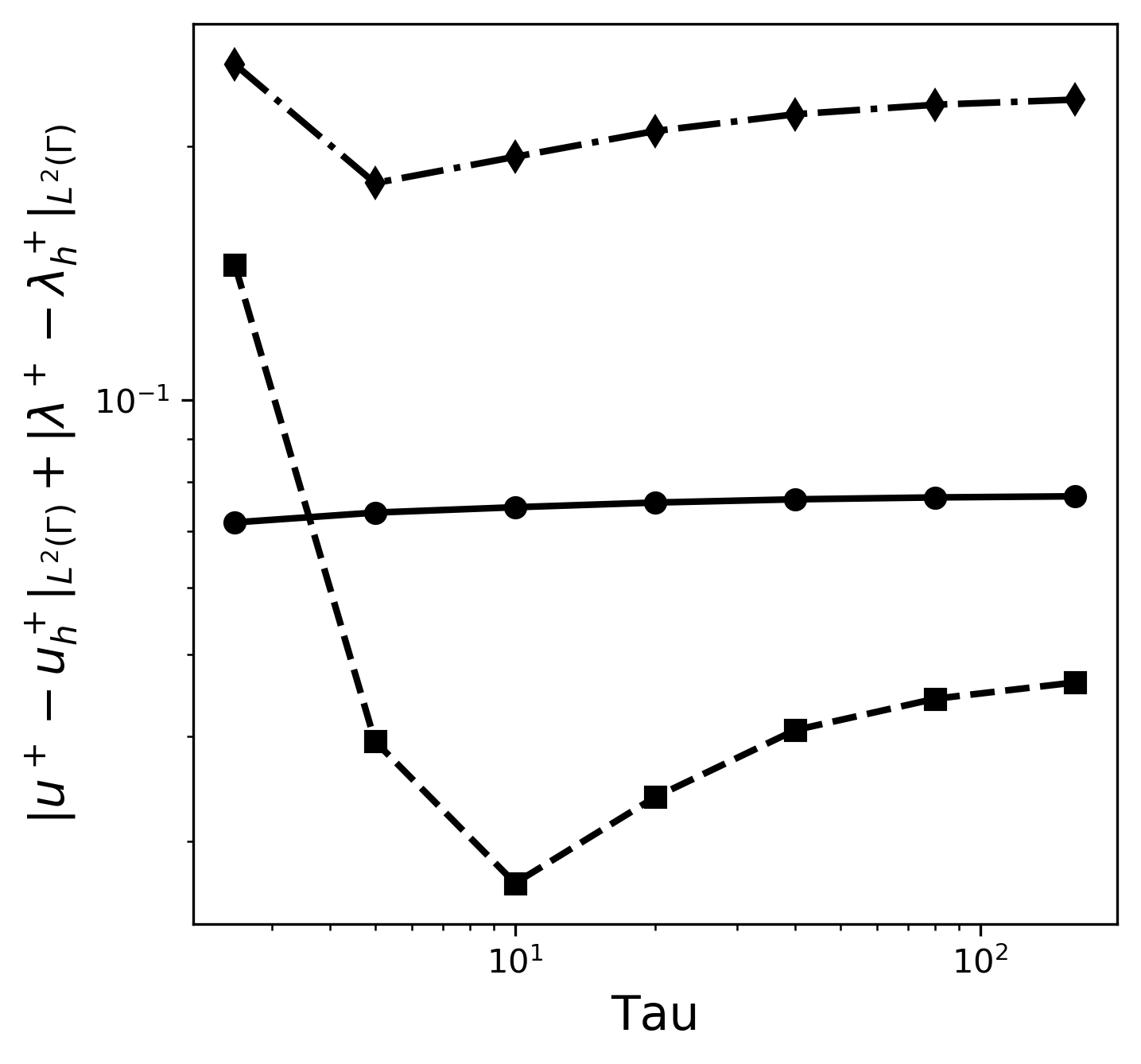}
}} \hfill
 \subfloat[Iteration taken by CG to solve the preconditioned system.
\label{fig:tau_iter}]{\resizebox {0.32\columnwidth} {!}{
    \includegraphics{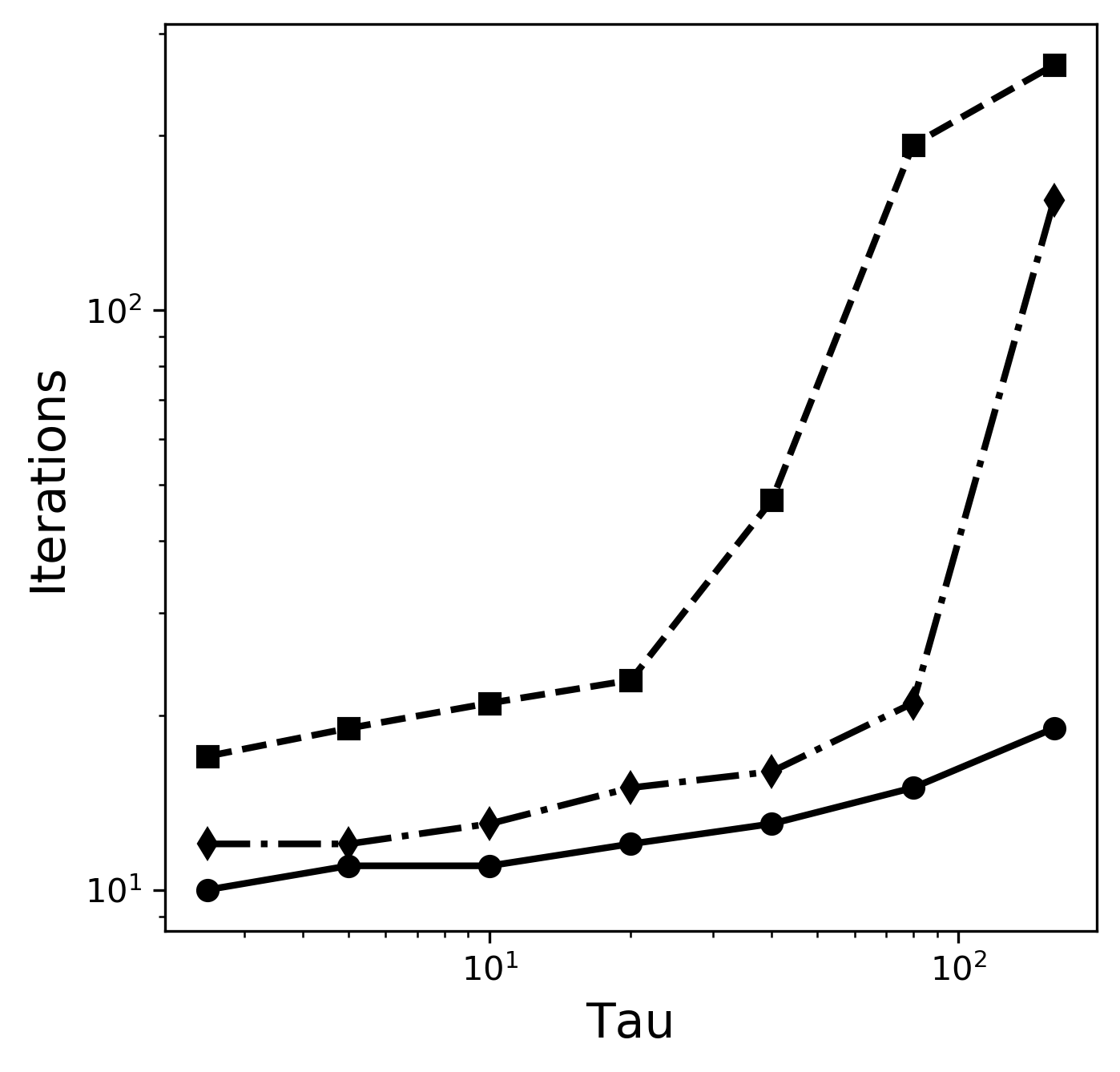}
 }} 
\caption{The dependence of the errors and iteration count on the value of $\tau$  for
$h < 2^{-1}$ (solid line with circles), $h <  2^{-2}$ (dash-dotted line with diamonds), and $h < 2^{-3.5}$ (dashed line with squares) for the problem on the unit sphere subdomain, with $j = k = m = l = 1$.}
\label{fig:tau_error}
 \end{figure} 
 
% \begin{figure}[!h] 
% \centering
%  \subfloat[Iteration taken by GMRES  to solve the preconditioned system.
% \label{fig:tau_iter}]{\resizebox {0.32\columnwidth} {!}{
%     \includegraphics{FEM_BEM_Fenics_CG_BEM_Mass_Laplace_sphere_P1_parameter_analyse_sphere_6_3_1_10-8_iter.png}
%  }} \hfill
%  \subfloat[Time to solve the preconditioned system.
% \label{fig:tau_time}]{\resizebox {0.32\columnwidth} {!}{
%     \includegraphics{FEM_BEM_Fenics_CG_BEM_Mass_Laplace_sphere_P1_parameter_analyse_sphere_6_3_1_10-8_time.png}
%  }} \hfill
% \caption{The dependence of the iteration count (bottom left)  and solving time ( bottom right) on the value of $\tau$  for
% $h < 2^{-1}$ (solid line with circles), $h <  2^{-2}$ (dash-dotted line with diamonds), and $h < 2^{-3.5}$ (dashed line with squares) for  problem with the unit sphere subdomain and  with $k = l = 1$.}
% \label{fig:tau}
%  \end{figure}
 
 Figure~\ref{fig:tau_error} shows the error values for different values of $\tau$ and $j = k = m = l = 1$. In this case, $\Gamma$  is smooth, and so $W_h^1 = \Lambda_h^1$. In Figure~\ref{fig:tau_FEM_error}, we plot in log-log scale the error of the interior solution $\|u^- -u_h^-\|_{L^2(\Omega^-)}$ and on Figure~\ref{fig:tau_BEM_error} the error of the exterior solutions $\|u^+ - u_h^+\|_{L^2(\Gamma)} + \|\lambda^+ - \lambda_h^+\|_{L^2(\Gamma)}$ for
$h < 2^{-1}$ (solid line with circles), $h <  2^{-2}$ (dash-dotted line with diamonds), and $h < 2^{-3.5}$ (dashed line with squares).
%Figure~\ref{fig:tau_iter} shows in log-log scale the dependence of number of iterations taken by GMRES on the value of $\tau$  and Figure~\ref{fig:tau_time} plots the time required to solve the preconditioned system obtained from the discrete problem~\eqref{eq:fem_bem_disctrete} for  
% $h < 2^{-1}$ (solid line with circles), $h <  2^{-2}$ (dash-dotted line with diamonds), and $h < 2^{-3.5}$ (dashed line with squares).

It can be seen from the Figures~\ref{fig:tau_error} %and~\ref{fig:tau} that the number of iterations, time  and 
that both errors stop decreasing when $\tau$ is around 10. Furthermore, for $\tau > 10$ the iterations increase with growing $\tau$, hence we fix $\tau = 10$ for the next experiments.

\subsection{Spherical subdomain}

Let $\Omega^-$ once again be the unit sphere, $\Gamma$  its boundary and consider the same exact solution as above. 
% We once again define 
% \begin{align*}
%  u^- (x, y, z) & = \frac{1}{2 \pi} \sin\left(\pi (x^2+y^2+z^2)\right)  + \frac{1}{2 \pi} \cos\left( \pi(x^2+y^2+z^2)\right) + \frac{2 \pi+1}{2 \pi}, \\
%  u^+ (x, y, z) & = \frac{1}{\sqrt{x^2+y^2+z^2}}.
% \end{align*}

\begin{figure}[!h] 
\centering
 \subfloat[Error of the interior (dashed line with circles) and exterior solutions (solid line with squares). The dotted line shows order 2 convergence.
\label{fig:sphere_convergence}]{\resizebox {0.32\columnwidth} {!}{
    \includegraphics{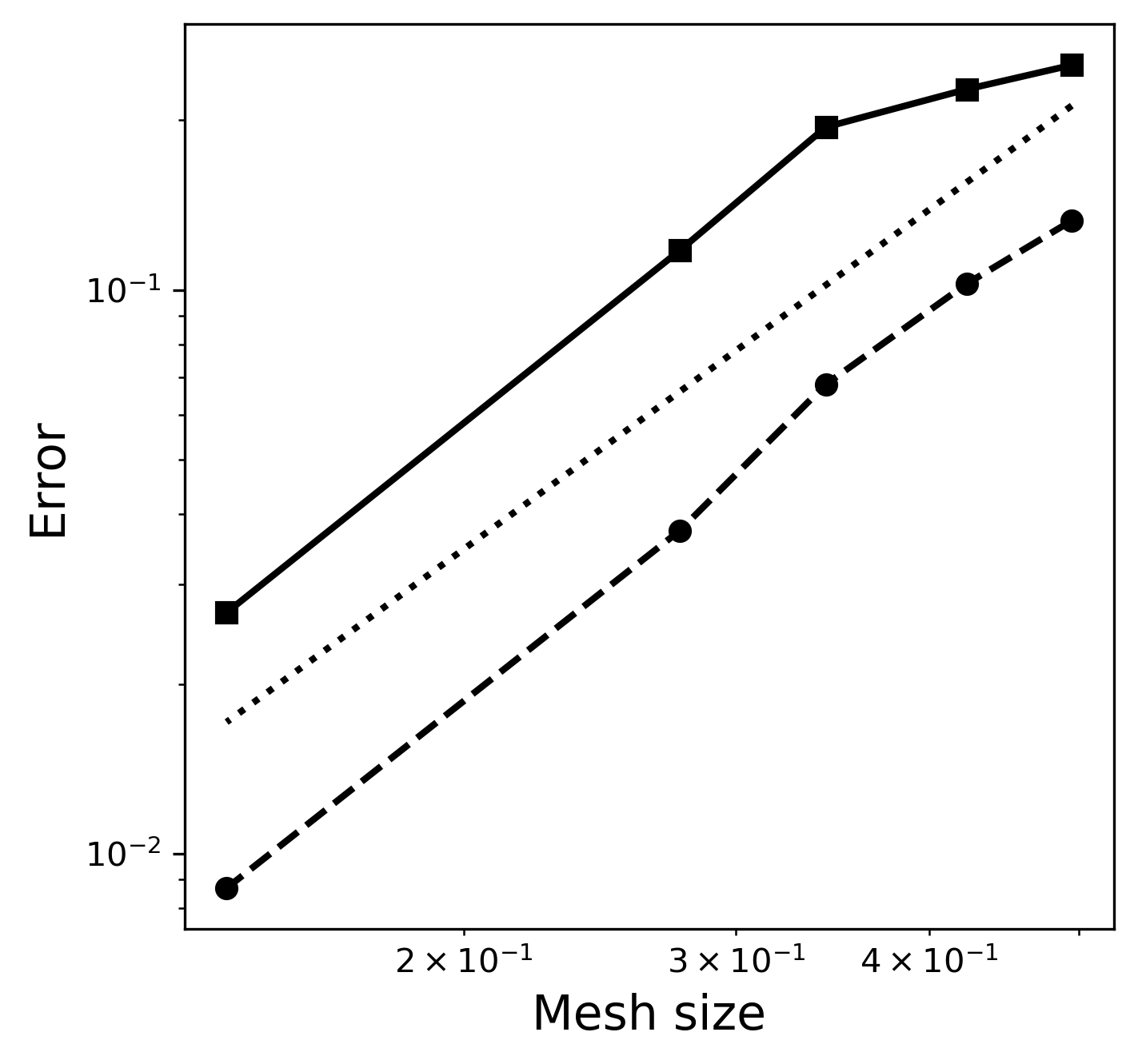}
}} \hfill
 \subfloat[Iteration taken by CG to solve the non-preconditioned system (dashed line), compared with the preconditioned system (solid line)
\label{fig:sphere_iter}]{\resizebox {0.32\columnwidth} {!}{
    \includegraphics{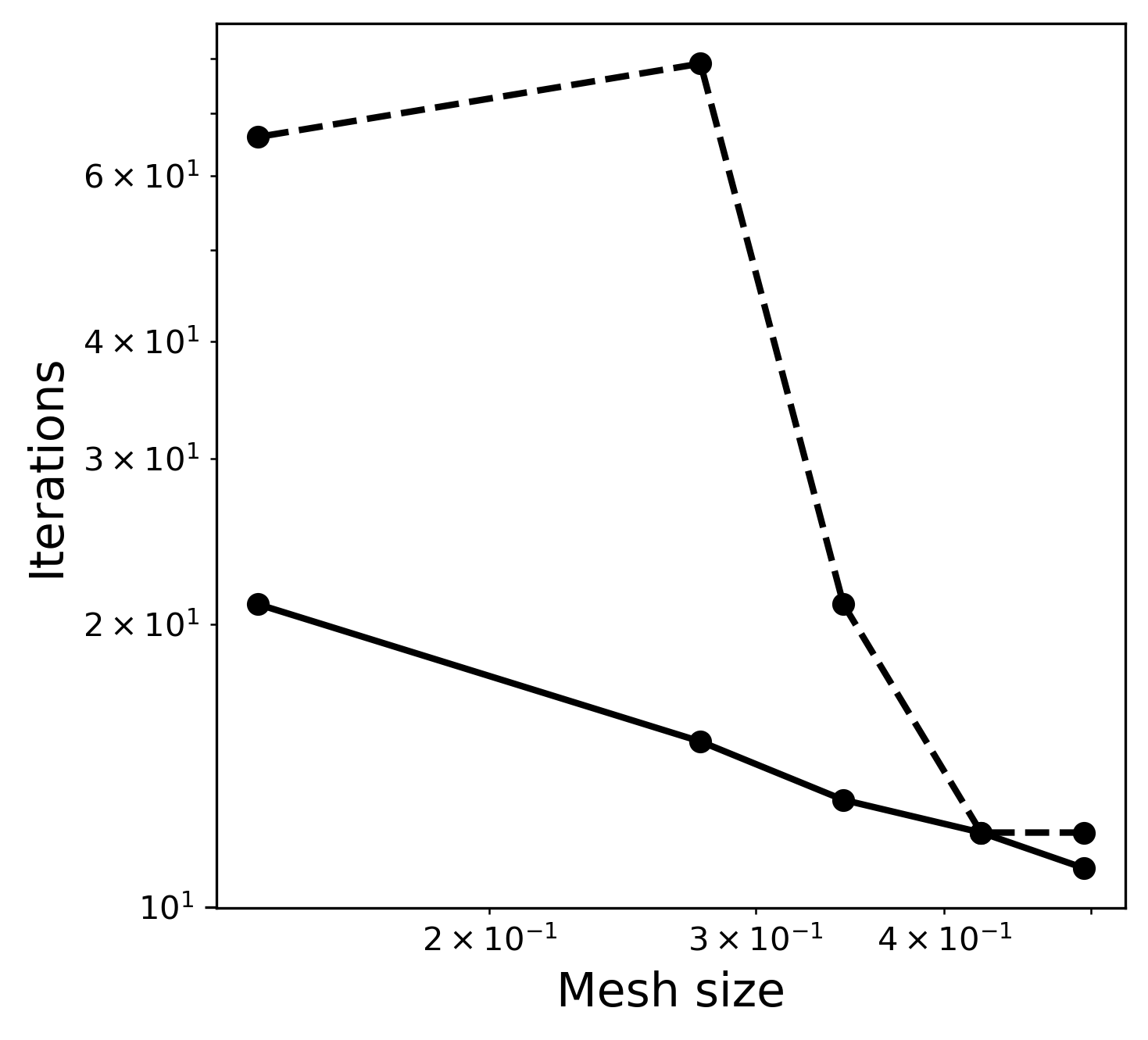}
 }} \hfill
 \subfloat[Time to solve the non-preconditioned system (dashed line), compared with the preconditioned system (solid line) of the whole discrete problem~\eqref{eq:fem_bem_disctrete}  including solving  interior and exterior systems.
\label{fig:sphere_time}]{\resizebox {0.32\columnwidth} {!}{
    \includegraphics{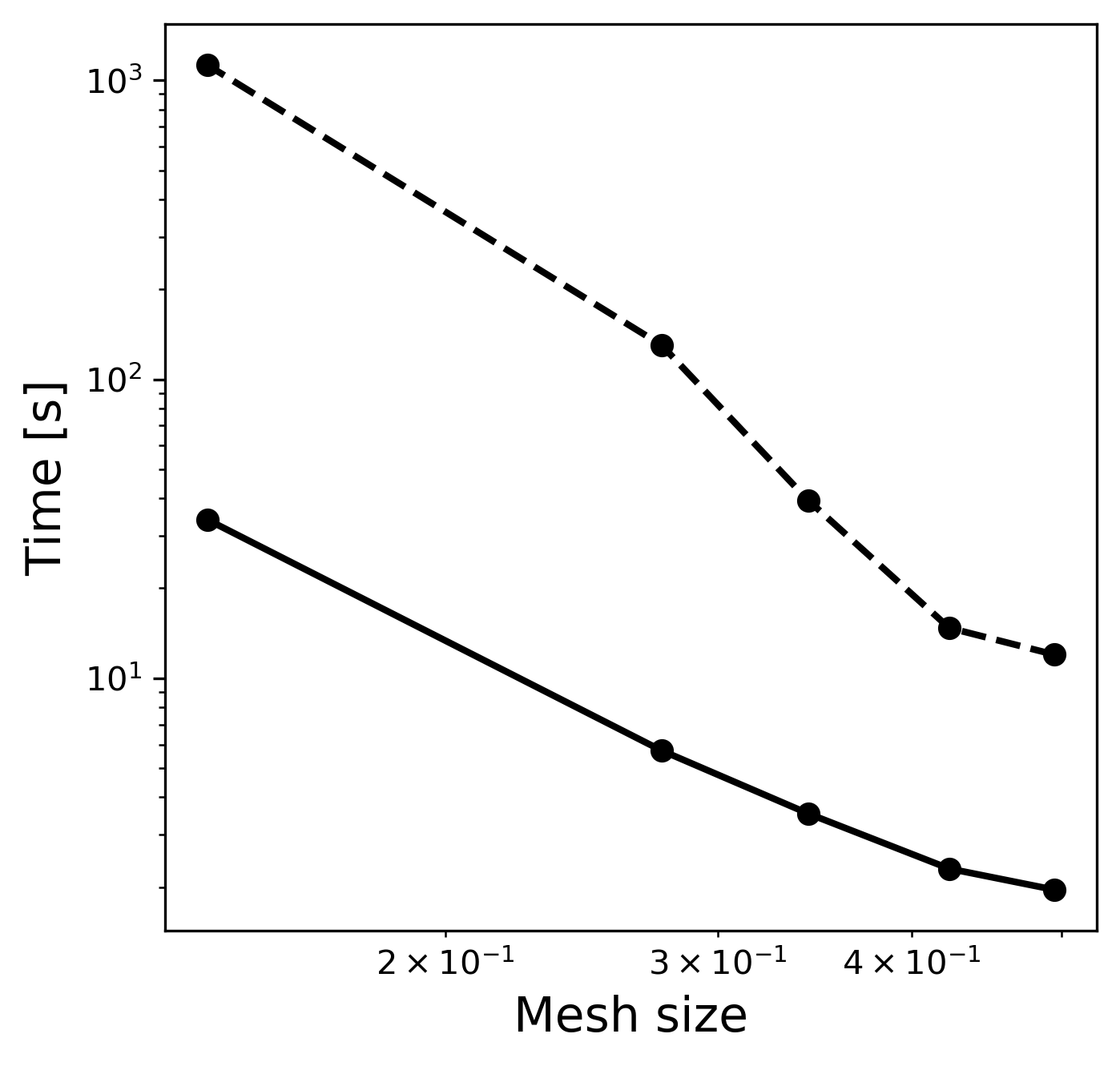}
 }} 
\caption{The convergence (left), CG iteration counts (middle) and solving time (right) for the problem on the unit sphere with  $\tau = 10$ and $j = k = m = l = 1$.}
\label{fig:sphere}
 \end{figure} 
Figure~\ref{fig:sphere} shows the convergence, CG iteration counts and solving time when $\tau = 10$ and $k = l = 1$. 
In Figure~\ref{fig:sphere_convergence}, we plot in log-log scale the error of the interior solution $\|u^- -u_h^-\|_{L^2(\Omega^-)}$ (dashed line with circles) and the error of the exterior solutions $\|u^+ - u_h^+\|_{L^2(\Gamma)} + \|\lambda^+ - \lambda_h^+\|_{L^2(\Gamma)}$ (solid line with squares).

 In Figure~\ref{fig:sphere_iter}, we plot in log-log scale the number of iterations taken by CG to solve the non-preconditioned system associated with exterior problem (dashed line), compared with the preconditioned system (solid line). In addition, Figure~\ref{fig:sphere_time} shows the time required by solvers of interior and exterior systems. The interior system is solved by CG with or without algebraic multigrid preconditioner and the exterior system is solved by GMRES with or without mass preconditioner. Preconditioning reduces both the iteration count and the CPU time needed by the solver.

\subsection{Cubical subdomain}
 
Let $\Omega^-= (0,1)^3$ be a cube and we solve the problem~\eqref{eq:divided_problem} with $f = 1$. We choose $\tau = 10$ and $j = k = m = l + 1 = 1$, where $\Lambda_h^0$ is the space of piece-wise constants per element in the trace space.

\begin{figure}[!h] 
\centering
 \subfloat[Error between interior solution and exterior. The dashed line shows order 2 convergence.
\label{fig:cube_convergence}]{\resizebox {0.32\columnwidth} {!}{
    \includegraphics{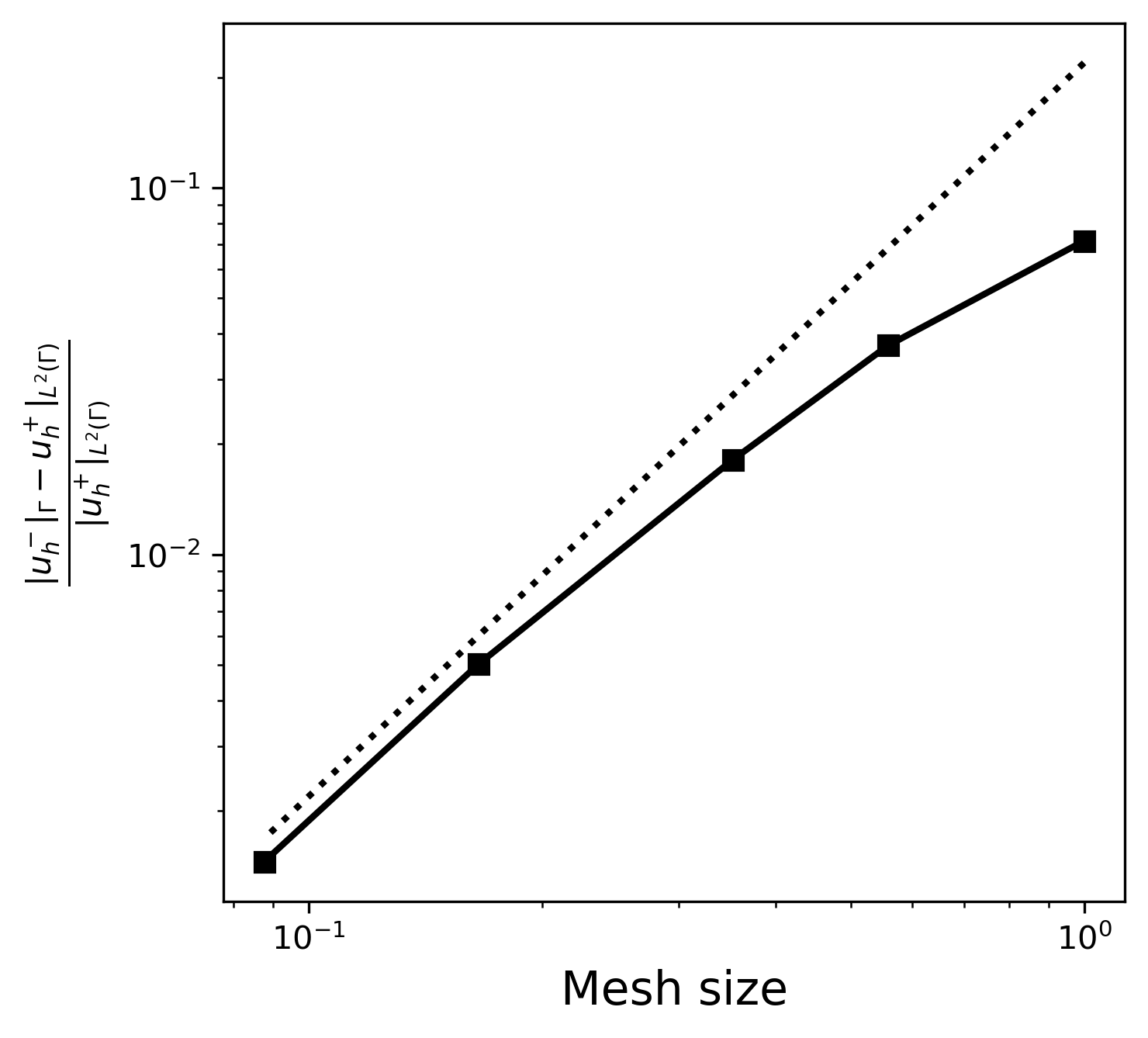}
}} \hfill
 \subfloat[Iteration taken by CG  to solve the non-preconditioned system (dashed line), compared with the preconditioned system (solid line)
\label{fig:cube_iter}]{\resizebox {0.32\columnwidth} {!}{
    \includegraphics{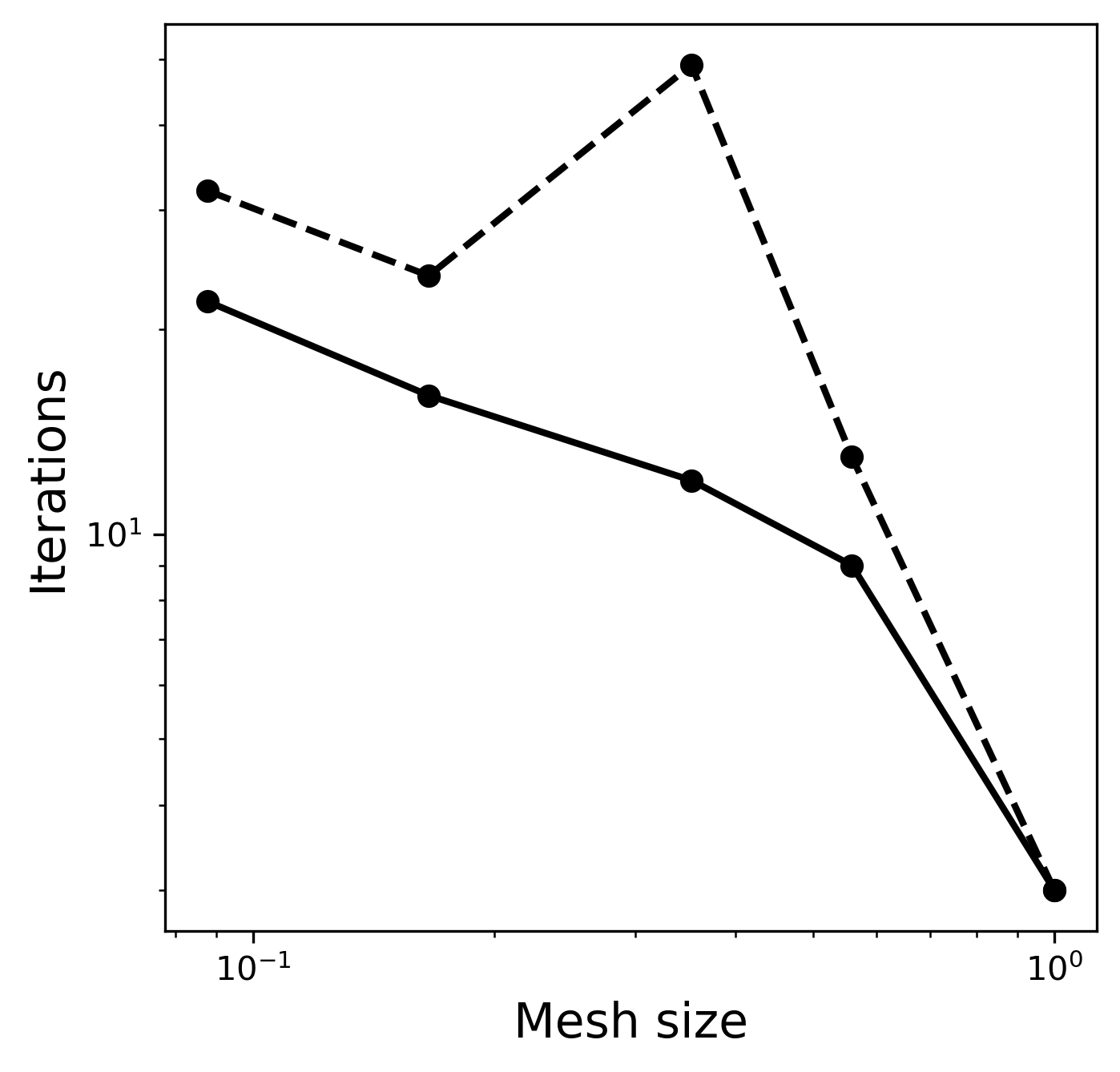}
 }} \hfill
 \subfloat[Time to solve the non-preconditioned system (dashed line), compared with the preconditioned system (solid line) of the whole discrete problem~\eqref{eq:fem_bem_disctrete}  including solving  interior and exterior systems.
\label{fig:cube_time}]{\resizebox {0.32\columnwidth} {!}{
    \includegraphics{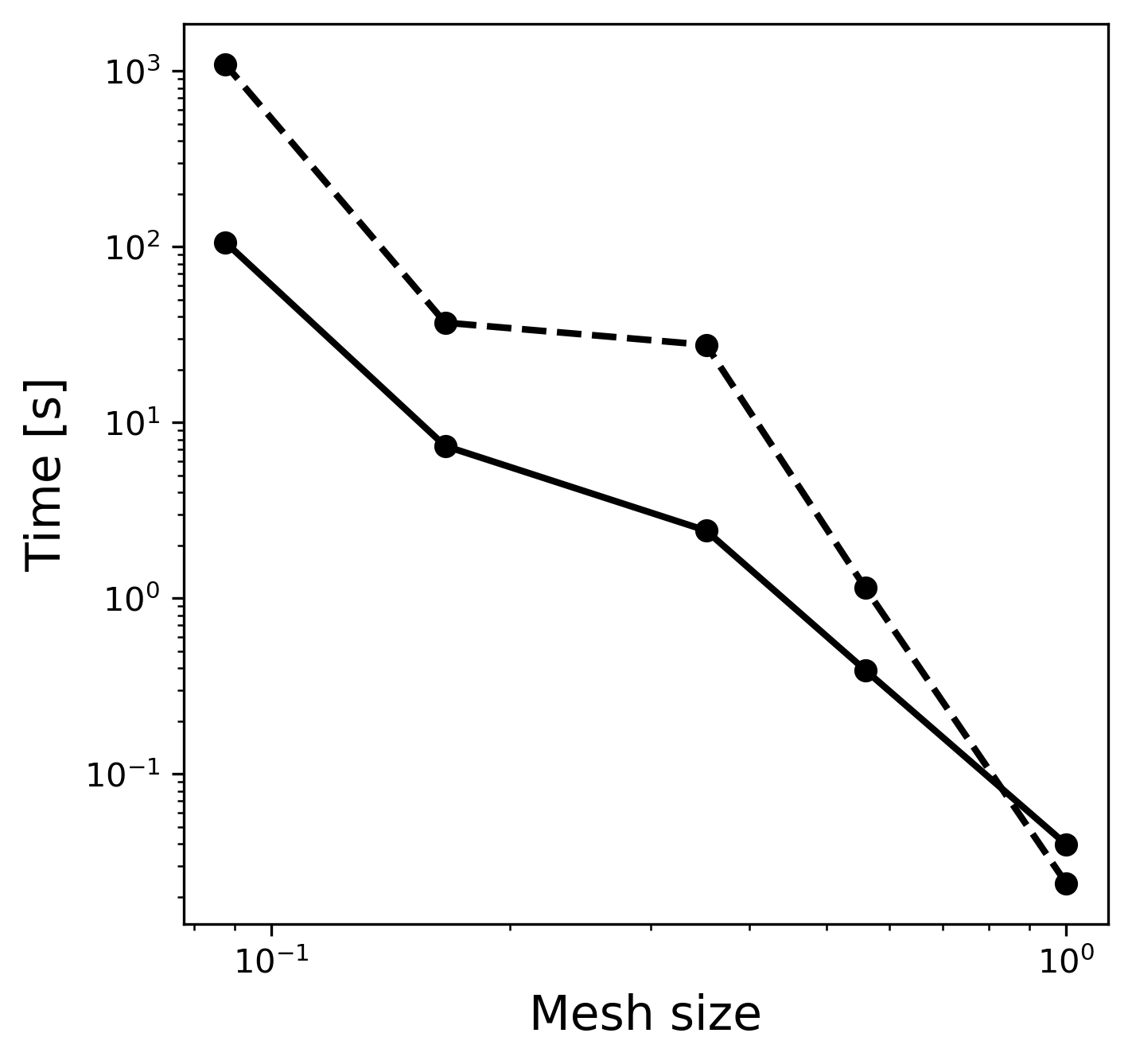}
 }} 
\caption{The convergence (left), CG iteration counts (middle) and solving time (right) for the problem on the cube with $\tau = 10$ and $j = k = m = l+1 = 1$.}
\label{fig:cube}
 \end{figure}

Figure~\ref{fig:cube_convergence} shows the convergence when $\tau = 10$ and $j = k = m = l+1 = 1$. In this case, the exact solution is not known, thus in Figure~\ref{fig:cube_convergence}, we plot in log-log scale the error between interior solution and exterior $\frac{\|u_h^-|_{\Gamma} -u_h^+\|_{L^2(\Gamma)}}{\|u_h^+\|_{L^2(\Gamma)}}$ (solid line).
 
 In Figure~\ref{fig:cube}, we plot as well in log-log scale the number of iterations and solving time taken by CG to solve the non-preconditioned system (dashed line), compared with the preconditioned system (solid line). Once again preconditioning brings improvement in terms of iteration counts and time taken to solve the problem.
 
 \section{Conclusions}
 \label{sec:conclusions}
 
 We have analyzed and demonstrated the effectiveness of Nitsche type methods for coupling finite element and boundary element formulations. Our approach gives flexibility to choose a continuous or discontinuous finite element space in the FEM solver, hence the interior problem can be solved essentially using any method that allows for the hybridised Nitsche method for interdomain coupling. We are also free to choose the trace variable minimising the coupling degrees of freedom.
 In this paper, we focus on the technical aspects and analysis to allow for flexibility within our framework. A work demonstrating the applicability for large problems using parallel approach is in preparation.
 
The method can be extended to other models such as the Helmholtz equations. In this case it is known~\cite{MR2793906} that the use of impedance interface conditions is advantageous and such an approach can be mimicked in the present framework by letting the stabilisation constant have non-zero imaginary part and depend on the wave number.
%$\frac{i \tau}{h}$ where $\tau$  
Formulations of the presented FEM/BEM coupling method to the Helmholtz and Maxwell problems are currently in preparation. For these cases however more effective operator preconditioning techniques for exterior problem are essential, especially for high frequency problems. Despite that, we expect that their implementation will be similar to the presented Laplace case.

%%%%%%%%%%%%%%%%%%%%%%%%%%%%%%%%%%%%%%%%%%%%%%%%%%%%%%%%%%%%%%%%%%%%%%%%%%%%%%%%%%%%%%%%%%%%%%%%%%%%%%%%%%%%
% BILBIOGRAPHY
%%%%%%%%%%%%%%%%%%%%%%%%%%%%%%%%%%%%%%%%%%%%%%%%%%%%%%%%%%%%%%%%%%%%%%%%%%%%%%%%%%%%%%%%%%%%%%%%%%%%%%%%%%%%

\thispagestyle{empty}
  \bibliographystyle{alpha}
\bibliography{FEM_BEM}

\end{document}